\documentclass[12pt]{article}
\usepackage{bbm}
 \usepackage{amssymb}
\usepackage{amssymb, amsthm, amsmath, amscd}
\usepackage[usenames,dvipsnames]{color}
\usepackage{soul} 
\usepackage{xcolor} 

\newtheorem{thm}{Theorem}[section]
\newtheorem{lem}[thm]{Lemma}
\newtheorem{prop}[thm]{Proposition}
\newtheorem{cor}[thm]{Corollary}
\theoremstyle{definition}

\theoremstyle{definition}
\newtheorem{df}[thm]{Definition}
\theoremstyle{definition}
\newtheorem{rem}[thm]{Remark}
\newtheorem{nota}[thm]{Notation}
\theoremstyle{definition}

\renewcommand{\phi}{\varphi}

\definecolor{purple}{RGB}{150,10,200} 

\newcommand{\CAs}{$C^*$-algebras}

\newcommand{\Ped}{{\rm Ped}}
\newcommand{\Cu}{{\rm Cu}}
\newcommand{\LAff}{{\rm LAff}}

\newcommand{\QT}{{\rm QT}}

\newcommand{\TQT}{\widetilde{\rm QT}}

\newcommand{\TT}{\widetilde{\rm T}}
\newcommand{\rT}{{\rm T}}

\newcommand{\what}{\widehat} 
\newcommand{\wh}{\widehat} 
\newcommand{\ol}{\overline}  
\newcommand{\nzero}{\setminus\{0\}}

\newcommand{\QTAW}{\overline{{\rm QT}(A)}^w}

\newcommand{\N}{\mathbb{N}}
\newcommand{\K}{\mathbb{K}}
\newcommand{\Z}{\mathbb{Z}}

\newcommand{\R}{\mathbb{R}}
\newcommand{\C}{\mathbb{C}}

\numberwithin{equation}{section}

\renewcommand{\theequation}{e\,\arabic{section}.\arabic{equation}}

\newcommand{\Aff}{\operatorname{Aff}}

\newcommand{\Her}{\mathrm{Her}}

\newcommand{\dt}{\delta}
\newcommand{\ep}{\varepsilon}

\newcommand{\rforal}{\,\,\,{\rm for\,\,\,all}\,\,\,}
\newcommand{\CA}{$C^*$-algebra}

\newcommand{\af}{{\alpha}}
\newcommand{\bt}{{\beta}}

\newcommand{\wtd}{\widetilde}

\newcommand{\beq}{\begin{eqnarray}}
\newcommand{\eneq}{\end{eqnarray}}

\usepackage{amsfonts}
\usepackage{mathrsfs}
\usepackage{textcomp}
\usepackage[all]{xy}
\usepackage{enumerate}
\usepackage{hyperref} 

\title{Divisibility and Real Rank Zero} 
\author{Xuanlong Fu\footnote{Department of Mathematics,  Tongji University. Email: xuanlongfu@tongji.edu.cn}}
\date{ }
\begin{document}

\maketitle

\begin{abstract}
Let $A$ be a simple separable exact $C^*$-algebra that has traces. 
We show the 
following 
existed regularity properties 
are equivalent:

\quad(1) $l^\infty(A)/J_A$ has real rank zero, 
where $J_A$ is the trace kernel ideal.

\quad(2) $A$ is tracially almost divisible. 

\quad(3) $A$ is tracially $m$-almost divisible for some $m\in\N\cup\{0\}.$

\quad(4) $A$ has tracial approximate oscillation zero. 

\quad(5) $A$ has Property (TM).

We also show that for  an algebraically simple separable 
stable rank one \CA\ $B$ with  non-empty  compact ${\rm T}(B)$ 
and locally finite nuclear dimension, 
its uniform tracial completion 
$(\ol B^{\rT(B)}, \rT(B))$ 
is hyperfinite, type ${\rm II_1},$ 
and isomorphic to  $({\cal R}_{\rT(B)},\rT(B))$. 
Furthermore, 
$\ol{B}^{{\rm T}(B)}$ 
is pure, 
has real rank zero and stable rank one, 
and satisfies $\rT (\ol B^{\rT(B)} )= \rT(B).$

Consequently,  
every  simple separable unital 
diagonal AH-algebra $V$
(e.g. Villadsen algebras of the first type) 
has the following 
tracial strict comparison: 
For every $a,b\in V_+,$ if $d_\tau(a)<d_\tau(b)$ 
holds for all traces $\tau\in\rT(V),$ 
then there is a 
sequence $\{r_n\}\subset V$ such that 
$\lim_n\|a-r_n^*br_n\|_{2,\rT(V)}=0.$

\end{abstract}

\section{Introduction} 

Divisibility properties play an important role in the regularity theory of $C^*$-algebras. One of the key features of well-behaved $C^*$-algebras is that they possess nice divisibility properties. 
One source of the modern study of divisibility in \CAs\ can be traced back to 
R{\o}rdam's work 
on 
the structure of UHF-stable $C^*$-algebras (\cite{RorUHF1, RorUHF2}). Subsequently, R{\o}rdam's ideas led to the concept of \emph{approximately divisible}, developed by 
Blackadar-Kumjian-R{\o}rdam in \cite{BKR}. 
In 2007, Lin proposed a tracial version of approximate divisibility \cite[Definition 5.3]{Lin07}. 
In 2013, Hirshberg and Orovitz \cite{HO} introduced tracial $\mathcal{Z}$-absorbing and, using techniques of Matui and Sato, showed that for simple, separable, nuclear, unital $C^*$-algebras, $\mathcal{Z}$-stability is equivalent to tracial $\mathcal{Z}$-absorbing --- a central divisibility property.
Since then, various notions of tracial approximate divisibility have been introduced by several authors, 
including 
Fu--Lin \cite{FL2020, FL2022}, 
Fu--Li--Lin \cite{FLL21}, 
Castillejos--Li--Szab{\'o} \cite{CLS}, and Amini--Golestani--Jamali--Phillips \cite{AGJP}.

The divisibility properties discussed above share a common theme: 
each permits, in one way or another, 
the  tracially large 
completely positive contractive 
(c.p.c.) order zero embedding of matrices of arbitrarily large rank into \emph{suitable} central sequence algebras. 
From this viewpoint, a weaker form of divisibility can be obtained by relaxing the centrality requirement, 
allowing such embeddings into \emph{suitable} 
sequence algebras 
without insisting that the image of the embeddings be central.
Among them are Winter's \emph{almost divisibility} and \emph{tracial almost divisibility} \cite{W12pure}, introduced to show that finite nuclear dimension implies $\mathcal{Z}$-stability; 
Property (TM), introduced by the author together with H.~Lin \cite[Definition 8.1]{FLosc} for the study of stable rank one,
tracial oscillation,  and the surjectivity of $\Gamma$.  
See also Robert--R{\o}rdam (\cite{RobRor}) for more on divisibility properties. 
These divisibility properties are all closely related to another important regularity property: the surjectivity of the canonical map 
$\Gamma: \Cu(A)\to {\rm LAff}_+\left(\TQT(A)\right).$

{{N.~P.~Brown first raised the question of when  $\Gamma$ is  surjective (see the remark after Question 1.1 of \cite{T20}).}} 
The surjectivity of $\Gamma$ 
plays a key role in the computation of the Cuntz semigroups of \CAs.
Brown--Dadarlat--Toms showed that $\Gamma$ is surjective for simple  unital exact finite ${\cal Z}$-stable \CAs\ \cite[5.5]{BPT}. 
Elliott--Santiago--Robert generalized this to 
simple finite ${\cal Z}$-stable \CAs\ \cite[6.8]{ERS}.
H.~Thiel showed that $\Gamma$ is surjective for simple separable unital  \CAs\ of stable rank one (\cite[Theorem 8.11]{T20}), without using strict comparison. 
Antoine--Perera--Robert--Thiel removed the unital assumption in Thiel's result (\cite[Theorem 7.13]{APRT}).

In order to study the relations between stable rank one and other regularity properties, we introduced the notion of tracial approximate oscillation zero in \cite{FLosc} (joint work with Lin)
and  showed  
that when strict comparison is provided, 
stable rank one, 
tracial approximate oscillation zero, 
and Property (TM)
are equivalent.

In another direction, L.~G.~Brown and G.~K.~Pedersen introduced the concept of \emph{real rank zero} \cite{BP91}, which is of  fundamental importance. S.~Zhang's work \cite{Z91}
shows that 
real rank zero implies tracial almost divisibility for simple \CAs. 
The converse of Zhang's theorem does not hold in general, due to the possible lack of projections in simple \CAs. 
Nevertheless, the tracial sequence algebra $l^\infty(A)/J_A$ could be rich in projections,
where $J_A$ is the trace kernel ideal. 
Indeed, our recent work \cite{Fu2025} shows that if we additionally assume that $A$ has stable rank one, then  $l^\infty(A)/J_A$ has real rank zero. 
One of our main results in this paper is that $l^\infty(A)/J_A$ has real rank zero if and only if $A$ is tracially almost divisible (see Theorem \ref{elmy11-10}). 

In \cite{Fu2025}, we introduced \emph{hereditary surjectivity} 
and showed that for simple separable C*-algebras, 
stable rank one implies hereditary surjectivity of $\Gamma,$ 
which in turn implies tracial approximate oscillation zero.
In this paper, we continue the study of the relation between tracial approximate oscillation zero and certain properties of the canonical map $\Gamma$. 
One of our main results is that for exact $C^*$-algebras, 
tracial approximate oscillation zero 
is equivalent to $\Gamma$ being hereditary dense,
which is 
a concept that blends uniform denseness 
with some flavor of strict comparison, 
even though strict comparison itself is absent
(see Definition \ref{X-dense} and Theorem \ref{elmy11-10}).

The last pieces to complete the puzzle are two divisibility properties: \emph{tracial diagonal divisibility} 
(see Definition \ref{elmy20-02}), 
meaning that every positive element is 
close to an 
evenly divided positive element with respect to the 2-norm; 
and 
\emph{weak tracial diagonal divisibility} 
(see Definition \ref{elmy20-01}), 
meaning that every positive element is 
close to an evenly divided positive element  
at the trace level.

With the tools of tracial approximate oscillation zero, 
hereditary denseness, 
tracial diagonal divisibility, 
and weak tracial diagonal divisibility, 
we can now unify various existing regularity properties
in the following theorem: 

\begin{thm}\label{elmy11-10}
Let $A$ be an algebraically simple separable 
non-elementary \CA\ with $\TQT(A)\neq\{0\}.$ Consider the following regularity properties: 

\quad$(1)$ 
$\Gamma: \Cu(A)\to\LAff_+(\TQT(A))$ is hereditary dense (see Definition \ref{X-dense}).

\quad$(2)$ $A$ has tracial approximate oscillation zero. 

\quad$(3)$ $l^\infty(A)/J_A$ has real rank zero. 

\quad$(4)$ $A$ is tracially diagonally divisible
(see Definition \ref{elmy20-02}). 

\quad$(5)$ $A$ has Property (TM). 

\quad$(6)$ $A$ is tracially almost divisible. 

\quad$(7)$ $A$ is tracially $m$-almost divisible 
for some $m\in\N\cup\{0\}.$

\quad$(8)$ $A$  is weakly tracially diagonally divisible 
(see Definition \ref{elmy20-01}).

Then $(1)\Rightarrow(2)\Leftrightarrow(3)\Rightarrow(4)\Rightarrow(5)\Rightarrow (6)\Rightarrow(7)$  always holds. 

If in addition  $A$ is exact, then 
$(7)\Rightarrow(8)\Rightarrow(1)$  holds 
and thus all properties are equivalent. 

\end{thm}

We should point out that some partial results are already known.  Assuming strict comparison, the equivalence of (2) and (5) was proved in \cite{FLosc}, while Lin \cite{LinJFA} showed that (6) is equivalent to (2) under the same assumption. 
Additionally, A.~Vaccaro's recent paper \cite{Vacc} contains the implication 
$(3)\Rightarrow (6)$ for the unital case.

As an application of Theorem \ref{elmy11-10}, 
we consider 
tracially complete \CAs\ 
--- a powerful tool that 
has been thoroughly  studied in \cite{TraCom}. 
We show (in Theorem \ref{elmy14-08}) 
that for a simple separable 
stable rank one \CA\ $B$ with  compact ${\rm T}(B)$ 
and locally finite nuclear dimension, 
its uniform tracial completion $(\ol B^{\rT(B)},\rT(B))$
is hyperfinite, type ${\rm II_1},$ 
and isomorphic to the hyperfinite model $({\cal R}_{\rT(B)},\rT(B))$. 
Furthermore, 
the uniform tracial completion  of $B$
is pure, 
has real rank zero and stable rank one, 
and satisfies $\rT (\ol B^{\rT(B)} )= \rT(B).$

Villadsen algebras of the first type are well known for their lack of strict comparison.
One interesting consequence of our result is that 
every  simple separable unital 
diagonal AH-algebra $V$ 
--- including  Villadsen algebras of the first type ---
has the following 
\emph{tracial strict comparison} (Corollary \ref{elmy15-03}): 
For every $a,b\in V_+,$ if $d_\tau(a)<d_\tau(b)$ 
holds for all traces $\tau\in\rT(V),$ 
then there is a 
sequence $\{r_n\}\subset V$ such that 
$\lim_n\|a-r_n^*br_n\|_{2,\rT(V)}=0.$

This paper is organized as follows. 
Section~2 sets up basic definitions. 
Section~3 studies tracial approximate oscillation zero and proves a decomposition theorem (Theorem~\ref{eler03-1}). 
Section~4 introduces hereditary denseness, proves a decomposition theorem (Theorem~\ref{20251216-1}), and shows that hereditary denseness implies tracial approximate oscillation zero. 
Section~5 shows that for a c.p.c.~order zero map $\phi: M_n \to A$, $\widehat{\phi(1_n)}$ can be approximated by $\widehat{[x]}$ for some $x$. 
Section~6 shows that strictly positive affine functions on traces are evaluations at positive elements. 
Section~7 introduces tracial diagonal divisibility and weak tracial diagonal divisibility, and shows that weak tracial diagonal divisibility implies hereditary denseness, while tracial approximate oscillation zero implies tracial diagonal divisibility. 
Section~8 shows that tracial almost divisibility implies weak tracial diagonal divisibility. 
Section~9 summarizes all the results and proves Theorem~\ref{elmy11-10}. 
Section~10 gives applications.

{\bf Acknowledgements.} 
This research was supported by
the Fundamental Research Funds for the
Central Universities. 

I would like to thank  
Guihua~Gong, Huaxin~Lin, 
Leonel~Robert, Aaron~Tikuisis, 
and Jianchao~Wu for helpful communications.

\section{Preliminaries}

\begin{nota} 
The set of all positive integers is denoted by $\N.$ 
The set of all non-negative real numbers is denoted by $\R_+.$
The set of all compact operators on a separable 
infinite-dimensional Hilbert 
space is denoted by $\K.$ 
Let $\{e_{i,j}\}$ denote a set of matrix units of $\K.$
%
Let  $(X,d)$ 
be a metric space,  let $x,y\in X,$ let $A,B\subset X,$ and let 
$\ep>0$.
We write $x\approx_{\ep}y$ if
$d(x,y)\le \ep$. 
We write $A\subset_\ep B$ if 
for all $a\in A,$ there is $b\in B$ with $a\approx_\ep b.$


\end{nota}

\begin{nota}
Let $A$ be a $C^*$-algebra. 
Denote by $A^{1}$ the closed unit ball of $A,$ and 
by $A_+$ the set of all positive elements in $A.$
Put $A_+^{1}:=A_+\cap A^{1}.$ 
The set of all self-adjoint elements of $A$ is denoted by $A_{sa}.$ 
Let $a\in A_+.$ Let $\Her_A(a)$ (or just $\Her(a),$ when $A$ is clear)
be the hereditary $C^*$-subalgebra of $A$ generated by $a.$
The Pedersen ideal of $A$ is 
denoted by $\Ped(A),$ which is the minimal dense ideal of $A$ (\cite[5.6]{Pedbk}).  
Let ${\rm Ped}(A)_+= {\rm Ped}(A)\cap A_+,$
${\rm Ped}(A)^{1}= {\rm Ped}(A)\cap A^{1}$ and ${\rm Ped}(A)_+^{1}={\rm Ped}(A)\cap A_+^{1}.$ 

Let $\ep >0.$ Define a continuous function
$f_{\ep} 
: \R
\rightarrow [0,1]$ as following: 
$f_\ep (t)=0$ for $t\in(-\infty,\ep],$ 
$f_\ep (t)=1$ for $t\in[2\ep,+\infty),$
and $f_\ep$ is linear on $[\ep,2\ep].$ 
\end{nota}

\begin{df}\label{Dcuntz}
Let $A$ be a \CA\
and let  $a, b\in A_+.$ 
We write $a \lesssim b$ if there are
$x_k\in A$
such that
$\lim_{k\rightarrow\infty}\|a-x_k^*bx_k\|=0$.
We write $a \sim b$ if $a \lesssim b$ and $b \lesssim a$  both hold. 
The Cuntz relation $\sim$ is an equivalence relation.
Set $\Cu(A)=(A\otimes \K)_+/\sim.$  
For $a\in (A\otimes \K)_+,$ let $[a]$ denote the Cuntz equivalence class corresponding to $a.$ The partial order on $\Cu(A)$ is given by the following: We write $[a]\le [b]$ whenever $a\lesssim b$ holds. 

Let $\iota: \K\otimes M_2(\C)\to \K$ be a $*$-isomorphism, which induces a $*$-isomorphism $\bar \iota:={\rm id}_A\otimes \iota: (A\otimes\K)\otimes M_2(\C)\to A\otimes\K.$ 
{{For $a,b\in (A\otimes \K)_+,$ define $[a]\oplus[b]:=[\bar \iota(a\otimes e_{1,1}+b\otimes e_{2,2})]\in\Cu(A).$}} With this (well-defined) addition, $\Cu(A)$ becomes a semigroup, which is called the Cuntz semigroup of $A.$ 
\end{df} 

\begin{df}
Let $A$ be a \CA. 
Let $\wtd A$ denote the minimal unitization of $A.$ 
$A$ is said to have stable rank one 
if the set of invertible elements in $\wtd A$ is dense in $\wtd A.$
$A$ is said to have real rank zero 
if the set of invertible self-adjoint elements in $\wtd A$
is dense in $\wtd A_{sa}.$
\end{df}

\begin{df}\label{shi01-10}
(\cite[II.1.1]{BH})
Let $A$ be a pre-\CA. 
A quasitrace on $A$ is a map $\tau: A\to\C$ 
such that 
{\bf (1)} $\tau(x^*x)=\tau(xx^*)\ge 0$ for all $x\in A$; 
{\bf (2)} $\tau$ is linear on commutative $^*$-subalgebras of $A$; 
{\bf (3)} $\tau(a+ib)=\tau(a)+i\tau(b)$ for all $a,b\in A_{sa}.$ 

If $\tau$ can be extended to a quasitrace on $M_2(A),$ then $\tau$ is called a 2-quasitrace. 
\end{df}

Adopting the convention in \cite[2.7]{FLosc} and \cite{Fu2025}, we have  the following: 
\begin{df}\label{Dqtr}
Let $A$ be a \CA. 
A densely  defined 2-quasitrace on $A\otimes \K$ is a 2-quasitrace 
$\tau:\Ped(A\otimes \K)\to\C.$ 
Denote by $\TQT(A)$ the set of all densely defined 2-quasitraces 
on $A\otimes \K.$  
Let $\TT(A):=\{\tau\in\TQT(A): \tau\text{ is linear}\}$
be the set of  traces on $\Ped(A\otimes\K).$ 

The partial order on $\TQT(A)$ is the canonical one: 
For $\tau_1,\tau_2\in\TQT(A),$ we write 
$\tau_1\le \tau_2$ if $\tau_1(a)\le \tau_2(a)$ for all $a\in \Ped(A\otimes\K)_+.$
The topology on $\TQT(A)$ is defined by pointwise convergence: 
A net $\{\tau_i\}\subset \TQT(A)$ is converge to $\tau\in\TQT(A)$
if and only if 
$\lim_i\tau_i(a)=\tau(a)$ for all $a\in\Ped(A\otimes\K).$ 
\end{df}

\begin{df}\label{DGamma}
Let $A$ be a \CA\ with $\TQT(A)\neq\{0\}.$
Denote by  $\Aff\left(\TQT(A)\right)$ the set of continuous real valued 
functions $f$  on $\TQT(A)$ 
such that $f(s \tau)=s f(\tau),$ 
$f(\tau+\sigma)=f(\tau)+f(\sigma)$
for all $s\in \R_+$ and $\tau,\sigma\in \TQT(A).$ 
{{Note that if $f\in \Aff\left(\TQT(A)\right),$ then $f(0)=0.$}} 
Moreover, define 
\beq
\hspace{-0.2in}
&\Aff_+\left(\TQT(A)\right):=\left\{f\in \Aff\left(\TQT(A)\right):  f(\tau)>0\mbox{ if }\tau\in \TQT(A)\setminus \{0\}\right\}\cup \{0\},&
 \nonumber\\
&{\rm LAff}_+\left(\TQT(A)\right):=
\left\{f:\TQT(A)\to [0,\infty]: \exists\ \{f_n\}\subset  \Aff_+\left(\TQT(A)\right)\mbox{ with } f_n\nearrow f\right\},&
\nonumber
\eneq 
where $f_n\nearrow f$ means for all $n\in\N$ and all $\tau,$ $f_n(\tau)\le f_{n+1}(\tau),$  and $f(\tau)=\lim_i f_i(\tau).$ 
\end{df}

\begin{df}
For $\tau\in  \TQT(A)$
and  $a\in (A\otimes\K)_+,$ 
we define 
$d_\tau(a):=\lim_{n}\tau(f_{1/n}(a)).$
\end{df} 

\begin{df}\label{shi28-1}
Let $A$ be a \CA. 
For $a\in {{(A\otimes\K)_+}},$ define a map 
$\wh{[a]}:\TQT(A) \to [0, \infty],$ 
$\tau\mapsto d_\tau(a),$ also define a map 
$\wh{a}: \TQT(A) \to [0, \infty],$ $\tau\to \lim_{n\to\infty}\tau((a-1/n)_+).$ 
{{When $a\in\Ped(A\otimes\K)_+,$ 
$\what a(\tau)=\tau(a)$ due to the fact that  quasitraces are automatically lower semicontinuous on the Pedersen ideal
(see \cite[Proposition 2.7]{Fu2025}).}} 
The  canonical map $\Gamma$ is defined as following: 
\beq
\Gamma: \Cu(A)\to {\rm LAff}_+\left(\TQT(A)\right),\quad 
[a]\mapsto \wh{[a]}.
\eneq

\end{df}

Other versions and notations of 2-quasitraces (functionals) appear in the literature (e.g., \cite{ERS}). 
The following commutative diagram clarifies the relationships among these different versions.
See Section 2 of \cite{Fu2025} for further details.

\begin{prop}
{\rm (Proposition 2.18 of \cite{Fu2025})}
\label{jiu25-T1}
Let $A$ be a simple  \CA. 
Then the following diagram is commute, and all the maps are 
ordered affine homeomorphisms. 
$$
\xymatrix{
        \TQT(A) \ar[dr]_{\Delta} \ar[rr]_{\chi} & & \QT_2(A)\setminus\{\tau_\infty\} \ar[dl]^{\Delta}\\
         &  {\rm F}(\Cu(A))\setminus\{d_{\tau_{\infty}}\}& 
    }.
$$
\end{prop}

\section{Tracial approximate oscillation zero}

\begin{df}
Let $A$ be a \CA. 
For $\tau\in \TQT(A)$ (defined on $\Ped(A\otimes\K)$)
and for a $C^*$-subalgebra $B\subset A\otimes\K,$ 
define 
$\|\tau|_B\|:=\sup\{\|\tau(b)\| :b\in \Ped(B)_+^1\}.$ 
We identify $A$ with $A\otimes e_{1,1}\subset A\otimes \K$
canonically, 
where $e_{1,1}$ is a rank one projection in $\K.$ 
Define 
$\QT(A):= 
\left\{\tau\in \TQT(A): \|\tau|_A\|=1\right\}.$ 
Let $\ol{\QT(A)}^w$ 
be the closure of $\QT(A)$ in $\TQT(A).$ 
Define $\rT(A):=\left\{\tau\in \QT(A): \tau\text{ is linear}\right\}.$
\end{df}

\begin{prop}\label{elmy11-11}
{\rm (\cite[Proposition 2.9]{FLosc})}
Let $A$ be an algebraically simple  \CA.
Then $\ol{\QT(A)}^w$ is compact and Hausdorff and $0\not\in \ol{\QT(A)}^w.$ 
\end{prop}

\begin{df}\label{D2norm}
Let $A$ be a \CA\ with  
$\QT(A)\neq\emptyset.$ 
For each 
$x\in (A\otimes\K)_+$ and each 
$\lambda\in \TQT(A),$
define $\|x\|_{2,\lambda}:=\sup_{\ep>0}\{\lambda((x^*x-\ep)_+)^{1/2}\}\in[0,\infty].$  
{{When $x\in\Ped(A\otimes\K),$ 
$\|x\|_{2,\lambda}=\lambda(x^*x)^{1/2}$ due to the fact that  quasitraces are automatically lower semicontinuous on the Pedersen ideal
(see \cite[Proposition 2.7]{Fu2025}).}} 
Define
$\|x\|_{{2}}=\sup\{\|x\|_{2,\tau}: \tau\in \ol{\QT(A)}^w\}\in  [0,\infty].$
Let $l^\infty(A)$ be the \CA\ of all norm bounded sequences of $A.$ Define 
$
J_A:=\{\{x_n\}\in l^\infty(A): \lim_{n\to\infty}\|x_n\|_{2}=0\}.
$
Define $c_0(A):=\{\{x_n\}\in l^\infty(A):\lim_{n\to\infty}\|x_n\|=0\}.$
\end{df} 

\begin{df}\label{DefOS1}
(\cite[Definition A.1]{eglnkk0}, \cite[Definition 4.1]{FLosc})
Let $A$ be 
{{a \CA\ with}} $\QT(A)\neq \emptyset.$ 
Let $a\in  (A\otimes {\cal K})_+,$
define the tracial oscillation of $a$ on $S$ as following: 
\beq
\omega(a):=\lim_{n\to\infty}\sup\left\{d_\tau(a)-\tau(f_{1/n}(a)): \tau\in \ol{\QT(A)}^w\right\}.
\eneq
\end{df}

\begin{rem}
{\bf (1)} In \cite[Definition A.1]{eglnkk0}, $\omega(a)$ is defined by using traces, and the notation used there is $\omega_S(a).$
{\bf (2)} The notation of tracial oscillation used in  \cite[Definition 4.1]{FLosc} is $\omega(a)|_S.$ 
\end{rem}

\begin{df}\label{elmy13-04}
(\cite[Definition 4.7, Definition 5.1]{FLosc}) 
Let $A$ be 
{{a \CA}} with $\QT(A)\neq \emptyset.$ 
For 
$a\in {{(A\otimes\K)_+}},$
$a$ is said to have \emph{tracial approximate oscillation zero}, if for every $\ep>0,$ there is $c\in\Her_A(a)_+$ such that $\|a-c\|_{2}<\ep,$ $\|c\|\le \|a\|,$ and $\omega(c)<\ep.$ 

If $a$ has tracial approximate oscillation zero for all $a\in {{\Ped(A\otimes\K)_+}},$ then the \CA\ $A$ is said to have tracial approximate oscillation zero. 


Let $A$ be an  algebraically simple \CA\ with $\QT(A)\neq \emptyset$  and $A$ has tracial approximate oscillation zero. 
If $B$ is another \CA\ and  $A\otimes\K\cong B\otimes\K,$
then $B$ is also said to have tracial approximate oscillation zero. 
\end{df}

\begin{rem}
What we called tracial approximate oscillation zero here 
was called T-tracial approximate oscillation zero in \cite[Definition 5.1]{FLosc}.  There are also other variations of tracial approximate oscillation in \cite[Definition 4.7]{FLosc}. 
\end{rem}

The following are some frequently used properties of tracial oscillation: 

\begin{prop}\label{1204-2}
Let $A$ be 
a simple \CA\ with $\QT(A)\neq \emptyset.$ Let $a,b\in (A\otimes\K)_+.$

{\rm (i)} 
If $a\sim b,$ then $\omega(a)=\omega(b)$ 
{\rm (see \cite[Proposition 4.2]{FLosc}).} 

{\rm (ii)} 
If $ab=0,$ then $\omega(a+b)\le \omega(a)+\omega(b)$
{\rm (see \cite[Proposition 4.4 (2)]{FLosc}).} 

\end{prop}

The following proposition shows that 
tracial oscillation measures the distance from $\wh{[a]}$ to positive continuous affine functions.

\begin{prop}\label{20251221-1}
Let $A$ be an algebraically simple \CA\ with $\QT(A)\neq \emptyset$ and let $a\in \Ped(A\otimes\K)_+.$ Let $S:=\ol{\QT(A)}^w$ 
{{and $\Aff_+\left(S\right):=\{f|_S:f\in\Aff_+(\TQT(A))\}.$}}
Define 
\beq\label{20251212-1}
\af &:= &
\inf\left\{\sup\left\{|d_\tau(a)-h(\tau)|:\tau\in S\right\}:h\in \Aff_+\left(S\right)\right\}; 
\\
\bt &:=&
\inf\left\{\sup\left\{d_\tau(a)-h(\tau):\tau\in S\right\}:h\in \Aff_+\left(S\right),h\le \wh{[a]}\right\};\qquad 
\\
\gamma &:=&
\inf\left\{\sup\left\{d_\tau(a)-h(\tau):\tau\in S\right\}:h\in \Aff_+\left(S\right),h< \wh{[a]}\right\}.\qquad 
\eneq
Then 
$
\af\le \bt=\gamma=\omega(a)\le 2\af.
$
\end{prop}
\begin{proof}
Since $A$ is algebraically simple, 
$0\notin S$ 
(Proposition \ref{elmy11-11}). 
Then $\ep_0:=\inf\{d_\tau(a):\tau\in S\}>0.$ 
If $0$ is not an accumulate point of ${\rm sp}(a)\nzero,$ 
then $\wh{[a]}$ is continuous. 
Thus $\omega(a)=0$ and $0=\af=\bt=\gamma=\omega(a).$ 
In the following we may assume that $0$ is an accumulate point of ${\rm sp}(a)\nzero.$ 

It is trivial that $\af\le \bt\le \gamma.$ 
For each $n\in\N,$ define $h_n:S\to\R_+$ 
by $h_n(\tau):=\tau(f_{1/n}(a))$ for all $\tau\in S.$
Then $h_n\in \Aff_+\left(S\right)$ and  
$h_n< \wh{[a]}$ for all $n\in\N.$
By definition, we have 
\beq
\gamma\le \inf_n\left\{\sup\left\{d_\tau(a)-h_n(\tau):\tau\in S\right\}\right\}=\omega(a) . 
\eneq

Let $g\in \Aff_+\left(S\right)$ with $g\le \wh{[a]}$ and let $\lambda\in(0,1).$ Then $\lambda g\in \Aff_+\left(S\right).$ 
Note that for all $\tau\in S,$ $d_\tau(a)<\infty$
(see \cite[Proposition 2.10 (2)]{FLosc}). 
Then  for all $\tau\in S,$
$
\lambda g(\tau)<g(\tau)\le d_\tau(a)=\lim_n h_n(\tau).
$
Since $\{h_n\}$ is increasing and $S$ is compact, by a  Dini-type theorem (see \cite[Proposition 5.6]{Fu2025}), there is $n_0\in\N$ such that $\lambda g\le h_{n_0}.$ 
Then 
\beq
\hspace{-0.5in}
\omega(a)
&=&\inf_n\left\{\sup\left\{d_\tau(a)-h_n(\tau):\tau\in S\right\}\right\}\\
&\le &
\sup\left\{d_\tau(a)-h_{n_0}(\tau):\tau\in S\right\}
\\&\le & 
\sup\left\{d_\tau(a)-\lambda g(\tau):\tau\in S\right\} 
\\&= & 
\sup\left\{(d_\tau(a)- g(\tau))+(1-\lambda)g(\tau):\tau\in S\right\}. 
\\&\le & 
\sup\left\{d_\tau(a)- g(\tau):\tau\in S\right\}+
(1-\lambda)\sup\left\{g(\tau):\tau\in S\right\}. \ \ 
\eneq
Since $S$ is compact and $g$ is continuous,   $\sup\left\{g(\tau):\tau\in S\right\}<\infty.$ 
Let $\lambda\to 1,$ then we have 
$
\omega(a)\le \sup\left\{d_\tau(a)- g(\tau):\tau\in S\right\}.
$ Since $g$ is arbitrary (with $g\le \wh{[a]}$), 
we have $\omega(a)\le \bt.$
Then 
$
\omega(a)\le \bt\le\gamma\le \omega(a).
$ Thus 
\beq
\af\le \bt=\gamma= \omega(a).
\eneq

To show  $\omega(a)\le 2\af,$
let $\dt\in(0,\ep_0)$ and let $g\in  \Aff_+\left(S\right)$ such that 
\beq
\sup\left\{|d_\tau(a)-g(\tau)|:\tau\in S\right\}<\af+\dt.
\eneq 
Then $g(\tau)-\af-\dt<d_\tau(a)<g(\tau)+\af+\dt$ for all $\tau\in S.$ 
Note that $g-\af-\dt$ is continuous on $S,$ $\{h_n\}$ is an increasing sequence, and $g(\tau)-\af-\dt<d_\tau(a)=\lim_n h_n(\tau)$ for all $\tau\in S.$
By a Dini-type theorem (see \cite[Proposition 5.6]{Fu2025}), 
there is $n_1\in\N$ such that 
$g < h_{n_1}.$ Then for all $\tau\in S,$
$
d_\tau(a)-h_{n_1}(\tau)\le (g(\tau)+\af+\dt)-(g(\tau)-\af-\dt)=2\af+2\dt.
$
Then $\omega(a)\le \sup\{d_\tau(a)-h_{n_1}(\tau):\tau\in S\}\le 2\af+2\dt.$ Since  $\dt$ is arbitrary, we have $\omega(a)\le 2\af.$
\end{proof}

\begin{lem}\label{yi15-1gg} 
Let $A$ be a 
\CA\ with $\TQT(A)\neq \{0\}.$ 
Let $e\in \Ped(A\otimes\K)_+$ 
and let $\tau\in\TQT(A)$ with $\|\tau|_{\Her(e)}\|\le 1.$
Let $y\in\Her(e)_{sa}$
and $a,b\in \Her(e)_+^1.$ 
Then we have : 

(1) 
$\tau(|y|)\le\|y\|_{2,\tau}.
$


(2) 
$|\tau(a)-\tau(b)|\le 3\tau(|a-b|)^{1/2}
\le 3\|a-b\|_{2,\tau}^{1/2}. 
$ 
\end{lem}
\begin{proof}
(1): 
Let $\dt:=\|y\|_{2,\tau}.$ 
Since $x\le (x^2/\dt+\dt)/2$ holds for all $x\in\R,$ 
and since $\|\tau|_{\Her(e)}\|\le 1,$ 
we have 
$\tau(|y|)\le(\tau(y^2)/\dt+\dt)/2= (\|y\|^2_{2,\tau}/\dt+\dt)/2
=\|y\|_{2,\tau}.$ 

(2): Note that $\tau(b)\le 1$ and $\tau(|a-b|)\le 1.$ 
Then by \cite[Lemma 3.5 (1)]{Haa}, we have 
$\tau(a)=\tau(b+(a-b))\le \tau(b+|a-b|) 
\le (\tau(b)^{1/2}+\tau(|a-b|)^{1/2})^2\le \tau(b)+2\tau(|a-b|)^{1/2}+\tau(|a-b|)\le \tau(b)+3\tau(|a-b|)^{1/2}.$ 
Hence $|\tau(a)-\tau(b)|\le 3\tau(|a-b|)^{1/2}.$ 
The last  inequality in (2) follows from (1). 
\end{proof}

\begin{lem}\label{elmy07-01}
Let $A$ be a $\sigma$-unital 
\CA\ 
with $\TQT(A)\neq \{0\}.$ 
Let $e\in \Ped(A\otimes\K)_+$ and let  $a,b\in \Her(e)_+^1.$  
Let $\tau\in\TQT(A)$ with $\|\tau|_A\|\le 1.$

Then $M :=\sup\{\|\lambda|_{\Her(e)}\|:\lambda \in \ol{\QT(A)}^w\}<\infty,$ 
and $|\tau(a)-\tau(b)| \le 3M^{3/4} \|a-b\|_{2}^{1/2}.$
\end{lem}
\begin{proof}
By  \cite[Proposition 2.10 (2)]{FLosc}, 
$M <\infty.$
Let $s:=\|\tau|_{\Her(e)}\|$ and let 
$\tau_0:=\tau/s,$ then $\|\tau_0|_{\Her(e)}\|=1.$
By   Lemma \ref{yi15-1gg} (2), 
\beq
|\tau(a)-\tau(b)|/s
=|\tau_0(a)-\tau_0(b)| 
\le 3\|a-b\|_{2,\tau_0}^{1/2}
=3  \|a-b\|_{2,\tau}^{1/2}/s^{1/4}.
\eneq
Then $|\tau(a)-\tau(b)|\le 3 s^{3/4} 
\|a-b\|_{2,\tau}^{1/2}
\le 3 M^{3/4} \|a-b\|_{2,\tau}^{1/2}
\le 3 M^{3/4} \|a-b\|_{2}^{1/2}.$
\end{proof}


\begin{lem}\label{eler09-1}
Let $A$ be a $\sigma$-unital 
\CA\ with $\TQT(A) \neq \{0\}.$ 
Let $a\in \Ped(A\otimes\K)_+^1$ 
{{with tracial approximate oscillation zero.}}
Then for all  $\ep>0,$
there is $b\in\Her(a)_+^1$ such that 
$\tau(a)< \tau(b)+\ep$ and $d_\tau(b)-\tau(f_{1/4}(b))<\ep$ for all $\tau\in\TQT(A)$ with $\|\tau|_A\|\le 1.$
\end{lem} 
\begin{proof}
By  Lemma \ref{elmy07-01}, there is $s>0$ such that 
for all $\tau\in  \TQT(A)$ with $\|\tau|_A\|\le 1$ and all $x,y\in\Her(a)_+^1,$ 
\beq\label{elmy005-08}
\|\tau|_{\Her(a)}\|\le s \mbox{\quad and\quad}
|\tau(x)-\tau(y)|\le s \|x-y\|_2^{1/2}.
\eneq 

Let $\ep>0$
and $\dt:=\min\{\ep/8, (\ep/4s)^2\}.$
Since $a$ has tracial approximate oscillation zero, 
there is $c\in\Her(a)_+^1\setminus\{0\}$ such that $\omega(c)<\dt$ and $\|a-c\|_2<\dt.$ 
By \eqref{elmy005-08}, 
\beq
\tau(a)
\le \tau(c)+s \|a-c\|_2^{1/2}
\le \tau(c)+s\dt^{1/2}
\le \tau(c)+\ep/4.
\eneq
Since $\omega(c)<\dt<\ep/4,$ there is $\theta>0$ such that 
$\theta<\min\{\ep/4(s+1),\|c\|/4\}$ and 
$\sup\{d_\tau(c)-\tau(f_\theta(c)):\tau\in\ol{\QT(A)}^w\}<\ep/4.$ 
Let $b:=f_{\theta/4}(c).$
Let $\tau\in \ol{\QT(A)}^w.$ Then 
\beq
d_\tau(b)-\tau(f_{1/4}(b))
=d_\tau(f_{\theta/4}(c))-\tau(f_{1/4}(f_{\theta/4}(c)))
\le d_\tau(c)-\tau(f_\theta(c))\le \ep/4, 
\eneq
and $
\tau(a)\le \tau(c)+\ep/4 \le\tau((c-\theta)_+)+\theta s+\ep/4
\le \tau(f_{\theta/4}(c))+\theta s+\ep/4\le \tau(b)+\ep/2.  
$
\end{proof}

\begin{lem}\label{elmy05-04}
Let $A$ be 
a simple   \CA\ with $\QT(A)\neq\emptyset$ 
and let $e\in \Ped(A\otimes\K)_+.$ 
Assume that for every $\ep>0,$ 
there is a sequence of mutually orthogonal positive elements $\{b_n\}\subset \Her(e)_+^1$ such that 
$ \sum_{n=1}^\infty \omega(b_n)<\ep$ and $d_\tau(e)-\sum_{n=1}^\infty d_\tau(b_n)<\ep$ for all $\tau\in \ol{\QT(A)}^w.$ 
Then $e$ has tracial approximate oscillation zero. 
\end{lem}
\begin{proof}
Let $S:=\ol{\QT(A)}^w$
and let $\ep>0.$ Let $\{b_n\}_{n\in\N}$ be as in the statement.  
Since $\sum_{n=1}^\infty\omega(b_n)<\ep,$ there are $s_n\in\Her(b_n)_+^1$ for all $n\in\N$ such that 
\beq\label{elsan28-07}
\sum_{n=1}^\infty\sup\{d_\tau(b_n)-\tau(s_n):\tau\in S\}<\ep. 
\eneq
Then for all $\tau\in S,$ we have 
\beq\label{elsan28-06}
d_\tau(e)<
\sum_{n=1}^\infty d_\tau(b_n)+\ep
<\sum_{n=1}^\infty \tau(s_n)+2\ep. 
\eneq
For all $n\in\N,$ let $d_n:=e^{1/2}(\sum_{m=1}^n s_m)e^{1/2}
\le e.$ 
By \cite[Proposition 5.5]{Fu2025},  
for all $\tau\in S,$
\beq\label{elsan28-08}
\lim_{n\to\infty}\tau(e-d_n)\le 
\lim_{n\to\infty}
\left(d_\tau(e)-\tau\left(\sum_{m=1}^n s_m\right)\right)
\overset{\eqref{elsan28-06}}{<}2\ep.
\eneq 
By a Dini-type result (see \cite[Proposition 5.6]{Fu2025}), there is $n_0\in\N$ such that 
$
\tau(e-d_{n_0})\le 3\ep \mbox{ for all }\tau\in S.
$ Thus 
\beq\label{elsan28-09}
\|e-d_{n_0}\|_2
\le (3\ep\cdot \|e\|)^{1/2}. 
\eneq
Note that $d_{n_0}\sim \sum_{m=1}^{n_0} s_m$ and $\{s_m\}$ are mutually orthogonal. 
By Proposition \ref{1204-2},
\beq
\omega(d_{n_0})
&=&\omega\left(\sum_{m=1}^{n_0} s_m\right)
\le \sum_{m=1}^{n_0} \omega(s_m)
\le 
\sum_{m=1}^{n_0}\sup\{d_\tau(s_m)-\tau(s_m):\tau\in S\}
\nonumber\\&\le& 
\sum_{m=1}^{n_0}\sup\{d_\tau(b_m)-\tau(s_m):\tau\in S\}
\overset{\eqref{elsan28-07}}{<}\ep. \label{elsan28-10}
\eneq
Then  \eqref{elsan28-09} and \eqref{elsan28-10} show that $e$ has tracial approximate oscillation zero. 
\end{proof}


\begin{thm}\label{eler03-1} 
Let $A$ be a {{$\sigma$-unital algebraically simple}} \CA\ with $\QT(A)\neq\emptyset.$ 
Then the following are equivalent: 

${\rm (i)}$ $A$ has tracial approximate oscillation zero.

${\rm (ii)}$ For all $e\in \Ped(A\otimes\K)_+$ and all $\ep>0,$ 
there is a sequence of mutually orthogonal positive elements $\{b_n\}\subset \Her(e)_+^1$ such that 
$ \sum_{n=1}^\infty \omega(b_n)<\ep,$ and $d_\tau(e)-\sum_{n=1}^\infty \tau(b_n)<\ep$ for all $\tau\in \ol{\QT(A)}^w.$

${\rm (iii)}$ For all $e\in \Ped(A\otimes\K)_+$ and all $\ep>0,$ 
there is a sequence of mutually orthogonal positive elements $\{b_n\}\subset \Her(e)_+^1$ such that 
$ \sum_{n=1}^\infty \omega(b_n)<\ep,$ and $d_\tau(e)-\sum_{n=1}^\infty d_\tau(b_n)<\ep$ for all $\tau\in \ol{\QT(A)}^w.$

\end{thm}

\begin{proof} 
Note that ${\rm (ii)}\Rightarrow {\rm (iii)}$ is trivial because 
$\tau(b_n)\le d_\tau(b_n).$ 

(iii) $\Rightarrow$ (i): Let $e\in\Ped(A\otimes\K)_+.$ 
Since (iii) holds, by Lemma \ref{elmy05-04}, $e$ has tracial approximate oscillation zero. Since $e$ is arbitrary, 
(i) holds.

${\rm (i)}\Rightarrow {\rm (ii)}:$
Without loss of generality, we may assume that $\|e\|=1.$
Let 
$S:=\ol{\QT(A)}^w.$ 
Let $e_n:= f_{1/2^{n}}(e)\in\Her(e)_+^1$ $(n\in\N).$ 

\bigskip
{\bf Claim:} There
is a set of positive elements 
$\{c_{n,i}\in \Her(e)_+^1:n\in\N,  i\ge n\}$ 
such that for all $n\in\N,$ 
the following hold: 

$(1_n)$ $c_{n,n}\in \Her(e_{n})_+^1,$
{{and $c_{n,i+1}=f_{1/8}(c_{n,i})$ for all $n\in\N,  i\ge n;$}}

$(2_n)$ $c_{1,n}, c_{2,n},...,c_{n,n}$ 
are mutually orthogonal; 

$(3_n)$ $\sup\{d_\tau(c_{n,n})-\tau(f_{1/4}(c_{n,n})):\tau\in S\}< \ep/4^n;$ 

$(4_n)$ $
{{\tau}}(e_n 
)\le
(\sum_{m=1}^n d_\tau(c_{m,n})) 
+\ep/4^n +\ep/3$ for all $\tau\in S.$

\bigskip
{\bf Proof of the Claim:}
Since 
$c_{n,i+1}=f_{1/8}(c_{n,i})$ 
depend on $c_{n,n}$
for all $n\in\N,  i\ge n$ (see $(1_n)$),  
the key to proving the claim is 
to construct the sequence $\{c_{n,n}\}_{n\in\N}.$ 
In the following, 
we will construct $\{c_{n,n}\}$ inductively. 

For $n=1,$ by Lemma \ref{eler09-1},
there is $c_{1,1}\in\Her(e_1)_+^1$ such that 
$\tau(e)<\tau(c_{1,1})+\ep/4\le d_\tau(c_{1,1})+\ep/4$ for all $\tau\in S$ and 
$\sup\{d_\tau(c_{1,1})-\tau(f_{1/4}(c_{1,1})):\tau\in S\}< \ep/4.$
Thus $c_{1,1}$ satisfies $(1_1)-(4_1).$ 
Assume that for $n\ge 1,$ $\{c_{1,1},...,c_{n,n}\}$ have been constructed and satisfy 
$(1_n)-(4_n).$ 
Define  
\beq
d_n:=
e_{n+1}-f_{1/16}\left(\sum_{m=1}^n c_{m,n}\right).
\eneq
Since $e_{n+1}\cdot \left(f_{1/16}\left(\sum_{m=1}^n c_{m,n}\right)\right)=\left(f_{1/16}\left(\sum_{m=1}^n c_{m,n}\right)\right),$ we have 
\beq\label{elmy04-2} 
d_n\cdot \left(f_{1/16}\left(\sum_{m=1}^n c_{m,n}\right)\right)
=\left(f_{1/16}\left(\sum_{m=1}^n c_{m,n}\right)\right)\cdot d_n.
\eneq 
Let $m\le n.$
By $(1_n)$ 
we have 
$c_{m,n+1}\in \Her(c_{m,m})\subset \Her(e_m)\subset \Her(e_n).$ 
Then 
\beq\label{elsan28-01}
e_{n+1}\cdot\left(\sum_{m=1}^n c_{m,n+1}\right)=\sum_{m=1}^n c_{m,n+1}. 
\eneq
We also have 
\beq\label{elsan28-02}
\hspace{-0.3in}
f_{1/16}\left(\sum_{m=1}^n c_{m,n}\right)\cdot 
\left(\sum_{m=1}^n c_{m,n+1}\right)
&\overset{(1_n),(2_n)}{=} &
f_{1/16}\left(\sum_{m=1}^n c_{m,n}\right)\cdot 
f_{1/8}\left(\sum_{m=1}^n c_{m,n}\right)
\nonumber\\
&=&
f_{1/8}\left(\sum_{m=1}^n c_{m,n}\right)=
\sum_{m=1}^n c_{m,n+1}.
\eneq
By \eqref{elsan28-01} and \eqref{elsan28-02}, 
\beq\label{elsan27-01}
d_n\perp \sum_{m=1}^n c_{m,n+1}.
\eneq 

Since $A$ has tracial approximate oscillation zero, 
by Lemma \ref{eler09-1},
there is 
\beq\label{elsan27-02}
c_{n+1,n+1}\in\Her(d_n)_+^1\subset \Her(e_{n+1})_+^1
\eneq
such that 
\beq\label{elsan27-05}
\tau(c_{n+1,n+1})>\tau(d_n)-\ep/ 
4^{n+1}{\rm \ for\ all\ }
\tau\in S, {\rm \ and}
\eneq
\beq\label{elsan27-03}
\sup\{d_\tau(c_{n+1,n+1})-\tau(f_{1/4}(c_{n+1,n+1})):\tau\in S\}< 
\ep/ 
4^{n+1}.
\eneq 
Note that \eqref{elsan27-02} shows that $(1_{n+1})$ holds. 
$(2_n)$ together with \eqref{elsan27-01} and \eqref{elsan27-02} show that $(2_{n+1})$ holds.
\eqref{elsan27-03} shows that $(3_{n+1})$ holds. 
Let  $\tau\in S,$ then 
\beq
\tau(e_{n+1}) 
&=&
\tau\left(d_n+f_{1/16}\left(\sum_{m=1}^n c_{m,n}\right)\right)
\\
&\overset{\eqref{elmy04-2}}{=}&
\tau\left(d_n\right)
+\tau\left(f_{1/16}\left(\sum_{m=1}^n c_{m,n}\right)\right)
\\
&\overset{(2_n)
}{=}&
\tau\left(d_n\right)+\left(\sum_{m=1}^n \tau(f_{1/16}(c_{m,n}))\right)
\\
&\overset{(3_n)}{\le}&
\tau\left(d_n\right)+
\left(\sum_{m=1}^n (d_\tau(c_{m,n+1})+\ep/4^{m})\right)
\\
&\le &\tau(d_n)+
\left(\sum_{m=1}^n d_\tau(c_{m,n+1})\right)
+\ep/3 
\\&\overset{\eqref{elsan27-05}}{<}&
\left(\sum_{m=1}^{n+1} d_\tau(c_{m,n+1})\right)+\ep/
4^{n+1}+\ep/3. 
\label{elsan27-04}
\eneq
Then \eqref{elsan27-04} shows condition $(4_{n+1})$ holds. 
By induction, the claim holds.

Now we define $b_n:=f_{1/4}(c_{n,n})$ for all $n\in\N,$
where $c_{n,n}$ are the positive elements in the above claim. 
Let $m,n\in\N$ with $m<n.$
Note that 
$b_m\cdot c_{m,m+1}=f_{1/4}(c_{m,m})\cdot f_{1/8}(c_{m,m})
=f_{1/4}(c_{m,m})=b_m.$ By induction, 
$b_m\cdot c_{m,i}=b_m$ for all $i>m.$ 
In particular,  $b_m\cdot c_{m,n}=b_m.$ Then 
$b_m\in\Her(c_{m,n}).$
By $(2_n)$ of the claim, 
$c_{m,n}\perp c_{n,n}.$ 
Then 
\beq\label{elsan28-03}
b_m\perp b_n. 
\eneq
By $(3_n)$ of the claim,  for all $\tau\in S,$ 
\beq
\hspace{-0.2in}
d_\tau(b_n)-\tau(b_n)
= d_\tau(f_{1/4}(c_{n,n}))-\tau(f_{1/4}(c_{n,n}))
\le d_\tau(c_{n,n})-\tau(f_{1/4}(c_{n,n}))\overset{(3_n)}{<}\ep/4^n.
\eneq
Hence $\omega(b_n)\le \ep/4^n,$ and then 
\beq\label{elsan28-04}
\sum_{n=1}^\infty \omega(b_n)\le \sum_{n=1}^\infty \ep/4^n<\ep. 
\eneq

Let $\tau\in S$ and let $n\in\N.$ 
By $(3_n)$ and $(4_n)$ of the claim, we have 
\beq
\tau(e_{n} 
)
&\overset{(4_n)}{\le} &
\left(\sum_{m=1}^n d_\tau(c_{m,n})\right)+\ep/4^n
+\ep/3
\\ &\le &   
\left(\sum_{m=1}^n d_\tau(c_{m,m})\right)+\ep/4^n
+\ep/3
\\ &\overset{(3_n)}{\le} &
\left(\sum_{m=1}^n (\tau(f_{1/4}(c_{m,m})) +\ep/4^m)\right)+\ep/4^n
+\ep/3 
\\ &=  &
\left(\sum_{m=1}^n \tau(b_m)\right)+
\left(\sum_{m=1}^n \ep/4^m\right)+\ep/4^n+\ep/3
\\ &\le  &
\left(\sum_{m=1}^n \tau(b_m)\right)+11\ep/12.
\eneq
Then 
\beq\label{elsan28-05}
d_\tau(e)=\lim_{n\to\infty} \tau(e_{n})\le 
\lim_{n\to\infty} 
\left(\sum_{m=1}^n \tau(b_m)\right)+11\ep/12
<\sum_{m=1}^\infty \tau(b_m)+\ep. 
\eneq
Then \eqref{elsan28-03}, \eqref{elsan28-04}, and \eqref{elsan28-05} show that (ii) holds. 
\end{proof}

\section{Hereditary dense canonical map}

Recall from \cite{Fu2025} the concept of hereditary surjectivity of the canonical map $\Gamma:$
\begin{df}\label{HeDe}
(\cite[Definition 4.5]{Fu2025})
Let $A$ be a \CA\ with $\TQT(A)\neq\{0\}.$ 
We say the 
canonical map $\Gamma: \Cu(A)\to \LAff_+\left(\TQT(A)\right)$
is hereditary surjective,
if the following holds: 
For every $a\in (A\otimes\K)_+$ and every $f\in \LAff_+\left(\TQT(A)\right)$
with $f(\tau) <d_\tau(a)$ for all $\tau\in \TQT(A)\nzero,$ 
there is $b\in (A\otimes\K)_+$ such that 
$b\lesssim a$ and $d_\tau(b)=f(\tau)$ for all $\tau\in \TQT(A).$
\end{df}

Recall the following results from our previous work \cite{Fu2025}
that shows the relationships between stable rank one,
hereditary surjective, and tracial approximate oscillation zero. 
In this section, we will generalize the concept of 
hereditary surjective to hereditary dense (Definition \ref{X-dense}), 
and show that hereditary dense implies tracial approximate oscillation zero. Later, we will build the equivalence between 
hereditary dense implies tracial approximate oscillation zero (Theorem \ref{elmy11-10}). 

\begin{thm}
{\rm (\cite[Theorem 4.8, Theorem 5.8]{Fu2025})}
Let $A$ be a separable
simple \CA. 

(1) If $A$ has stable rank one, 
then $\Gamma$ is hereditary surjective. 

(2) 
If $\Gamma$ is hereditary surjective, 
then $A$ has tracial approximate oscillation zero. 
\end{thm}

\begin{df}\label{X-dense}
Let $A$ be a \CA\ with $\TQT(A)\neq\{0\}.$ 
Let $X\subset \Aff_+(\TQT(A))$  be a subset that containing $0$ and let 
\beq
\Sigma(X) &:=&\{h\in\LAff_+(\TQT(A)):
\mbox{ There is a sequence } \{h_n\} \subset X
\nonumber\\&& \quad 
\mbox{ such that } h(\tau)=\sum_{n=1}^\infty h_n(\tau) \mbox{ for all }\tau\in\TQT(A)\}.
\eneq
We say the 
canonical map $\Gamma: \Cu(A)\to \LAff_+\left(\TQT(A)\right)$
is $\Sigma(X)$-hereditary dense, if the following holds: 
For every $a\in (A\otimes\K)_+,$ every $\ep>0,$ 
and every $f\in \Sigma(X)$ 
with $f(\tau) <d_\tau(a)$ for all $\tau\in \TQT(A)\nzero,$ there is $b\in (A\otimes\K)_+$ such that 
$b\lesssim a$ and $|f(\tau)-d_\tau(b)|<\ep$ for all $\tau\in \TQT(A).$


Let $X_1:=\{\hat e\in \Aff_+(\TQT(A)): e\in \Ped(A\otimes\K)_+\}.$ 
If $\Gamma$ is $\Sigma(X_1)$-hereditary dense, 
then we say $\Gamma$ is \emph{hereditary dense}. 
In this case, $\Sigma(X_1)\supset 
 \left\{\wh e, \wh{[e]}:e\in \Ped(A\otimes\K)_+\right\}.$

\end{df} 

The following is taken from \cite[Lemma 5.3]{Fu2025}. 
See also \cite[3.1]{LinJFA}
and \cite[2.30]{LinAdv}.
\begin{lem}\label{yi04-1}
{\rm (Lin's orthogonality lemma, 
\cite[Lemma 5.3]{Fu2025})} 
Let $A$ be an 
algebraically  simple
\CA. Let $a\in (A\otimes\K)_+,$
$e\in \Ped(A\otimes \K)_+,$ and $\ep>0.$
Assume that $a\lesssim e$
and $\omega(a)<\ep.$  
Then 
there are $b,c\in \Her(e)_+^1$ and 
 such that   

(1) $d_\tau(a)-\ep<
\tau(b)\le d_\tau(b) \le d_\tau(a)$ for all 
$\tau\in \overline{\QT(A)}^w;$  
%

(2) $\omega(b)<\ep;$

(3) $bc=0;$

(4) 
$d_\tau(e)-d_\tau(a)\le d_\tau(c)<d_\tau(e)-d_\tau(a)+\ep
\label{yi04-4}
$ for all $\tau\in  \overline{\QT(A)}^w.$
\end{lem}

\begin{thm}\label{20251216-1}
Let $A$ be a non-elementary algebraically simple \CA\ with $\QT(A)\neq\emptyset.$
Let $0\in X\subset \Aff_+(\TQT(A))$ be a subset. 
Then the following are equivalent:

{\rm (1)} $\Gamma$ is $\Sigma(X)$-hereditary dense.

{\rm (2)} For every $e\in (A\otimes\K)_+,$ every $\ep>0,$ 
and every $h\in X$
with $h(\tau) <d_\tau(e)$ for all $\tau\in \TQT(A)\nzero,$ there is $y\in (A\otimes\K)_+$ such that 
$y\lesssim e$ and $|d_\tau(y)-h(\tau)|<\ep$ for all $\tau\in  \ol{\QT(A)}^w.$ 

{\rm (3)} 
For every $e\in \Ped(A\otimes\K)_+,$ 
every summable sequence $\{\ep_n\}\subset (0,+\infty)$ 
with $\ep:=\sum_{n=1}^\infty \ep_n\in(0,+\infty),$ 
and every $h\in \Sigma(X)$
with $h(\tau) <d_\tau(e)$ for all $\tau\in \TQT(A)\nzero,$
there is a sequence of mutually orthogonal positive elements $\{b_n\}\subset \Her(e)_+^1$ such that 
$ \omega(b_n)<\ep_n$ for all $n,$ and $
|h(\tau)-\sum_{n=1}^\infty d_\tau(b_n)|<\ep$ for all $\tau\in \ol{\QT(A)}^w.$

\end{thm}
\begin{proof} 
${\rm (3)\Rightarrow (1)}:$ 
Let $e, \ep, h,b_n$ be as in the statement of (3). 
Let $y:=\sum_{n=1}^\infty b_n/2^n\in\Her(e)_+^1.$
Then for all $\tau\in\ol{\QT(A)}^w,$
$d_\tau(y)=\sum_{n=1}^\infty d_\tau(b_n)\approx_\ep h(\tau).$ 
Thus (1) holds. 

${\rm   (1)\Rightarrow (2)}$ is trivial.

${\rm (2)\Rightarrow (3)}:$ 
Let $e\in (A\otimes\K)_+$ with $\|e\|=1$ and let $h\in \Sigma(X)$ satisfying $h(\tau) <d_\tau(e)$ for all $\tau\in \TQT(A)\nzero.$ 
Let $S:=\ol{\QT(A)}^w.$ 
Since $S$ is compact and $0<d_\tau(e)-h(\tau)$ for all $\tau\in S,$
there is $\dt>0$ such that 
\beq\label{1202-1} 
\dt < d_\tau(e)-h(\tau) \rforal  \tau\in S.
\eneq 
Let $\dt_n:=\min\{\ep_n/4^n, \dt/4^n\}.$ 
Since $h\in\Sigma(X),$ 
there is a sequence $\{h_i\}\subset X$ such that   
\beq\label{20251028-4}
h(\tau)=\sum_{i=1}^\infty h_i(\tau) \rforal \tau\in \TQT(A). 
\eneq

{\bf Claim:} \emph{Let $c_0:=e.$  There are sequences $\{a_n\}\subset (A\otimes\K)_+$
and $\{b_n\}, \{c_n\}\subset \Her(e)_+^1$ such that  for all $n\in\N,$} 

\emph{{\rm (i)} $a_n\lesssim c_{n-1}$ and $|d_\tau(a_n)-h_n(\tau)|\le \dt_n/4$ for all $\tau\in S;$}

\emph{{\rm (ii)} $b_n, c_n\in\Her(c_{n-1})_+^1,$ $b_nc_n=0,$ 
and $\omega(b_n)<\dt_n;$}

\emph{{\rm (iii)} $d_\tau(a_n)-\dt_n<d_\tau(b_n)\le d_\tau(a_n)$ for all $\tau\in S;$}

\emph{{\rm (iv)} $d_\tau(e)-\sum_{i=1}^nd_\tau(a_i)
\le d_\tau(c_n)
$ for all $\tau\in S.$ }

\bigskip
{\bf Proof of the Claim:} We will construct $a_n, b_n, c_n$ inductively. 
Since   $h_1\le h <\wh{[e]}$ holds, 
by (2), 
there exists $a_1\in (A\otimes\K)_+$
such that 

(${\rm i'}$) $a_1\lesssim e=c_0$ and $|d_\tau(a_1)-h_1(\tau)|\le \dt_1/4$ for all $\tau\in  S.$ 
\\
Note that $h_1\in X\subset \Aff_+(S).$  Then by (${\rm i'}$) and  Proposition \ref{20251221-1}, 
$
\omega(a_1)< \dt_1.
$
By 
Lemma  \ref{yi04-1}, 
there are $b_1,c_1\in \Her(c_0)_+$ such that 

(${\rm ii'}$) $b_1,c_1\in \Her(c_0)_+^1,$ $b_1c_1=0,$
 and $\omega(b_1)< \dt_1;$ 

(${\rm iii'}$) $d_\tau(a_1)-\dt_1<d_\tau(b_1)\le d_\tau(a_1)$ for all $\tau\in S;$

(${\rm iv'}$) $d_\tau(e)-d_\tau(a_1)
\le d_\tau(c_1)
$ for all $\tau\in S.$ 
\\
Hence the Claim holds for $n=1.$
Assume that for $n\in\N$  
we have constructed $\{a_i\},$ $\{b_i\},$ 
and $\{c_i\}$ ($1\le i\le n$) that
satisfy (i)-(iv). 
Let us proceed to the case $n+1:$ 
For all $\tau\in S,$ 
\beq\label{20251028-5}
h_{n+1}(\tau)
&\overset{\eqref{20251028-4}}{\le }& 
h(\tau)-\sum_{i=1}^n h_i(\tau)
\overset{\eqref{1202-1}}{<} 
d_\tau({{e}})
- \sum_{i=1}^n h_i(\tau)
-\dt 
\nonumber\\ 
&\overset{{\rm (i)}}{\le} & 
d_\tau({{e}})- \sum_{i=1}^n d_\tau(a_i)
\overset{{\rm (iv)}}{\le} 
d_\tau(c_n).
\eneq
\\
By \eqref{20251028-5} and (2),
there is $a_{n+1}\in (A\otimes\K)_+$
such that 

(${\rm i''}$) $a_{n+1}\lesssim c_n$ and $|d_\tau(a_{n+1})-h_{n+1}(\tau)|<\dt_{n+1}/4$ for all $\tau\in S.$ 
\\
Then by (${\rm i''}$) and  Proposition \ref{1204-2}, 
$\omega(a_{n+1})<\dt_{n+1}.$  By  Lemma \ref{yi04-1}, there are $b_{n+1},c_{n+1}\in \Her(c_n)_+$ such that 

(${\rm ii''}$) $b_{n+1}, c_{n+1}\in \Her(c_n)_+^1,$ 
$b_{n+1}c_{n+1}=0,$ and $\omega(b_{n+1})<\dt_{n+1};$ 

(${\rm iii''}$) $d_\tau(a_{n+1})-\dt_{n+1}<d_\tau(b_{n+1})\le d_\tau(a_{n+1})$ for all $\tau\in S;$

(${\rm IV}$) $d_\tau(c_n)-d_\tau(a_{n+1})
\le d_\tau(c_{{n+1}})
$ for all $\tau\in S.$ 
\\
Then for all $\tau\in S,$ 
\beq
\hspace{-0.3in}
&&d_\tau(e)-\sum_{i=1}^{n+1}d_\tau(a_i)
\overset{{\rm (iv)}}{\le}
d_\tau(c_n)-d_\tau(a_{n+1})
\overset{\rm (IV)}{\le} 
d_\tau(c_{{n+1}}).
\eneq
Hence we have 

(${\rm iv''}$) $d_\tau(e)-\sum_{i=1}^{n+1}d_\tau(a_i) 
\le d_\tau(c_{n+1})
$ for all $\tau\in S.$ 
\\
Then (${\rm i''}$), (${\rm ii''}$), (${\rm iii''}$), (${\rm iv''}$) show that 
$a_{n+1}, b_{n+1}, c_{n+1}$ satisfy the conditions of the Claim. 
By induction, the Claim holds.

Note that by (ii),  $\Her(c_0)\supset\Her(c_1)\supset\Her(c_2)\supset\cdots.$
Let $n,m\in\N$ with $n<m.$
By (ii), $b_m\in \Her(c_{m-1})\subset\Her(c_n)\subset\Her(b_n)^\perp.$ Hence 
\beq\label{20251210-3}
b_n\perp b_m \quad (n\neq m).
\eneq
Moreover, by (i) and (iii),  for all $\tau\in S$ and all $n\in\N,$
\beq
|h_n(\tau)-d_\tau(b_n)|
\le 
|h_n(\tau)-d_\tau(a_n)|+|d_\tau(a_n)-d_\tau(b_n)|
\overset{\rm (i),(iii)}{<}2\dt_n<\ep_n.
\eneq 
Hence 
\beq\label{20251210-2}
\left|h(\tau)-\sum_{n=1}^\infty d_\tau(b_n)\right| 
\le  
\sum_{n=1}^\infty \left|h_n(\tau)-d_\tau(b_n)\right|< 
\sum_{n=1}^\infty \ep_n=\ep. 
\eneq
By \eqref{20251210-3}, (ii), and \eqref{20251210-2}, we see that 
 ${\rm  (3)}$ holds. 
Then the  theorem holds.
\end{proof} 

\begin{cor}\label{elmy11-03}
Let $A$ be a non-elementary algebraically simple \CA\ with $\QT(A)\neq\emptyset.$
Assume that the canonical map 
$\Gamma$ is hereditary dense.
Then  $A$ has tracial approximate oscillation zero. 

\end{cor}

\begin{proof}
Let $X_1$ be as in Definition \ref{X-dense}. 
Let $e\in\Ped(A\otimes\K)_+$ and let $\ep>0.$ 
Let $S= \ol{\QT(A)}^w.$
Then $M:=\sup\{d_\tau(e):\tau\in S\}<\infty$ 
{{(\cite[
2.10 (2)]{FLosc}).}}
Let $c\in(0,1)$ such that $(1-c)M<\ep/4.$
Define a map $h: \TQT(A)\to \R_+, \tau\mapsto cd_\tau(e).$
Then $h\in \Sigma(X_1)$ and $h(\tau)<d_\tau(e)$ for all $\tau\in\TQT(A)\setminus\{0\}.$  
Since $\Gamma$ is hereditary dense, 
by (3) of Theorem \ref{20251216-1}, 
there is a sequence of mutually orthogonal positive elements $\{b_n\}\subset \Her(e)_+^1$ such that 
$ \omega(b_n)<\ep/2^{n+1}$ for all $n,$ and $
|h(\tau)-\sum_{n=1}^\infty d_\tau(b_n)|<\ep/2$ 
for all $\tau\in \ol{\QT(A)}^w.$
Then 
\beq
d_\tau(e)-\sum_{n=1}^\infty d_\tau(b_n)
\le |d_\tau(e)-h(\tau)|
+\left|h(\tau)-\sum_{n=1}^\infty d_\tau(b_n) \right|
<\ep
\quad
\left(\forall\ \tau\in \ol{\QT(A)}^w\right). 
\nonumber
\eneq
Then by (iii) of Theorem \ref{eler03-1}, $A$ has tracial approximate oscillation zero. 
\end{proof}

\section{Order zero map and tracial oscillation}
The standard reference for properties of a c.p.c.~order zero map is \cite{WZ09OZ}.
\begin{df} [\cite{WZ09OZ}]
Let $A,B$ be \CAs. 
A c.p.c.~map $\phi: A\to B$ is said to have order zero if $\phi(x)\phi(y)=0$ holds for all $x,y\in A_+$ with $xy=0.$  
\end{df}

\begin{nota}\label{elmy06-02}
Let $0<a<b.$ 
Define continuous function $\chi_{a,b}:[0,+\infty)\to[0,1]$ as following: 
$\chi_{a,b}(x)=0$ for $x\in[0,a],$ 
$\chi_{a,b}(x)=1$ for $x\in[b,\infty),$ 
and $\chi_{a,b}$ is linear on $[a,b].$
\end{nota} 

\begin{prop}\label{elmy06-03}
Let $n\in\N.$ For each $k\in\N\cup\{0\},$ 
define $\chi_k:=\chi_{\frac{k}{n},\frac{k+1}{n}},$ 
where $\chi_{\frac{k}{n},\frac{k+1}{n}}$ is defined in 
Notation \ref{elmy06-02}. Then 
$\sum_{k=m}^{n-1} \chi_k(x)=n (x-m/n)_+$ for all $x\in[0,1],$
where $m\in\{0,1,...,n-1\},$ and $(x-m/n)_+:=\max\{x-m/n, 0\}.$ 
\end{prop}
\begin{proof}
Trivial. 
\end{proof}


\begin{prop}\label{elmy06-01}
Let $A$ be a \CA\ 
with $\QT(A)\neq\emptyset$
and let $a\in\Ped(A\otimes\K)_+.$
Let $n\, (\ge 2)\in \N$
and let $\phi: M_n\to \Her(a)$ be a c.p.c.~order zero map. 
Let $\chi_k$ be as in Proposition \ref{elmy06-03}. 
Then for all $\tau\in\QTAW,$
%
\beq\label{elmy06-04}
\tau\left( \left(\phi(1_n)-\frac{1}{n}\right)_+\right)
=\tau\left(\sum_{k=1}^{n-1} \chi_k(\phi(e_{k,k}))\right)
\le d_\tau\left(\sum_{k=1}^{n-1} \chi_k(\phi(e_{k,k}))\right)
\le \tau(\phi(1_n)).
\nonumber
\eneq
Moreover, $\omega(\sum_{k=1}^{n-1} \chi_k(\phi(e_{k,k})))\le 1/n.$
\end{prop} 

\begin{proof}
Let $\tau\in\QTAW$ be arbitrary. 
For $k=1,...,n,$ let $a_k:=\phi(e_{k,k}),$ 
which are  mutually orthogonal and mutually (Murray-von Neumann) equivalent positive contractions. Then 
$\|\sum_{k=1}^{n-1} \chi_k(a_k)\|\le 1$
and $\tau\left(\sum_{k=1}^{n-1} \chi_k(a_k)\right)
\le d_\tau\left(\sum_{k=1}^{n-1} \chi_k(a_k)\right).$ 
We have 
\beq
\tau\left(\sum_{k=1}^{n-1} \chi_k(a_k)\right)
&=&
\sum_{k=1}^{n-1} \tau\left(\chi_k(a_k)\right)
=
\sum_{k=1}^{n-1} \tau\left(\chi_k(a_1)\right)
= 
\tau\left(\sum_{k=1}^{n-1}  \chi_k(a_1)\right)
\\&=&
\tau\left( n\left(a_1-\frac{1}{n}\right)_+\right)
= 
\tau\left( \left(\phi(1_n)-\frac{1}{n}\right)_+\right),
\mbox{ and }
\eneq

\beq
d_\tau\left(\sum_{k=1}^{n-1}
\chi_k(a_k)\right)
&= 
&
\sum_{k=1}^{n-1}d_\tau\left(
\chi_k(a_k)\right)
=
\sum_{k=1}^{n-1}d_\tau\left(
\chi_k(a_1)\right)
\\&\le &
\sum_{k=0}^{n-2}\tau\left(
\chi_k(a_1)\right)
= \tau\left(\sum_{k=0}^{n-2}
\chi_k(a_1)\right)
\\&\le  &
\tau\left(\sum_{k=0}^{n-1}
\chi_k(a_1)\right)
= \tau\left(na_1\right)
\\&=  &
\tau\left(\phi(1_n)\right). 
\eneq 
Moreover, 
$d_\tau(\sum_{k=1}^{n-1} \chi_k(a_k))-\tau(\sum_{k=1}^{n-1} \chi_k(a_k))
\le \tau(\phi(1_n)-(\phi(1_n)-1/n)_+)\le 1/n$
for arbitrary  $\tau\in\QTAW.$ Thus 
$\omega(\sum_{k=1}^{n-1} \chi_k(a_k))\le 1/n.$
%
\end{proof}

\section{Affine functions on traces are evaluations} 



In this section we work with traces rather than 2-quasitraces, since Cuntz--Pedersen's result \cite{CP79} holds for traces. 
The following proposition  
removes the unital condition in 
\cite[Corollary 3.10]{BPT}, 
which is an analog of 
Kadison's function representation theorem 
(see \cite[Theorem 3.10.3]{Pedbk}, see also 
\cite[Lemma 2.1]{APRT-JFA}). 

\begin{prop}\label{elmy18-01}
{\rm (cf. \cite[3.10]{BPT})}
Let $A$ be a 
\CA\ with 
$\rT(A)\neq \emptyset.$ 
Let $\rT_{[0,1]}(A):=\{s\tau: s\in[0,1],\tau\in\rT(A)\}.$
Let $\Aff(\rT_{[0,1]}(A))$ be the Banach space of $\R$-valued, continuous, affine functions on $\rT_{[0,1]}(A),$ equipped with the supremum norm.

Then for all $f\in \Aff(\rT_{[0,1]}(A)),$
there is $e\in A_{sa}$ such that $f(\tau)=\tau(e)$ for all $\tau\in\rT(A).$
%
If in addition, $A$ is algebraically simple and 
$f(\tau)>0$ for all $\tau\in \rT(A),$ 
then such $e$ can be chosen to be positive. 
\end{prop}
\begin{proof}
If $A$ is  unital, then the  lemma holds by 
\cite[Corollary 3.10]{BPT}. 
In the following we may assume $A$ is non-unital. 
Let $\wtd A$ be the minimal unitization of $A.$ 
For each $\tau\in \rT(A),$ let $\wtd \tau$ be the canonical extension of $\tau$ on $\wtd A$ 
given by 
$\wtd\tau(x+s)=\tau(x)+s$ for all $x\in A$ and all $s\in\C.$   
In the following we may regard 
$\rT(A)$ as a subspace of $\rT(\wtd A)$ 
by identifying $\tau\in\rT(A)$ with $\wtd \tau.$ 

Let $\pi:\wtd A\to \wtd A/A\cong \C$ be the quotient map. 
Then 
$\pi\in\rT(\wtd A)$ and $\pi|_A=0.$ 
For all $\tau\in \rT(\wtd A),$ 
if $\tau|_A=0,$ then $\tau=\pi;$
if $\tau_1:=\tau|_A\neq 0,$ then $\tau=\|\tau_1\|\cdot(\tau_1/\|\tau_1\|)+(1-\|\tau_1\|)\pi.$
Hence for all $\tau\in \rT(\wtd A),$ 
if $\tau\neq \pi,$ 
there are  $\tau_0\in\rT(A)$
and $s\in(0,1]$ such that $\tau=s\tau_0+(1-s)\pi.$
Furthermore, it is routine to check that  such $\tau_0$ and $s$ are unique.

Let 
$f\in\Aff(\rT_{[0,1]}(A)).$ 
Define $\wtd f: \rT(\wtd A)\to\R$ by 
$\wtd f(\pi):=0$ and 
$\wtd f(s\tau+(1-s)\pi):=f(s\tau)=sf(\tau)$ for all $s\in(0,1]$ and all $\tau\in\rT(A).$
Then $\wtd f$ is an affine function on $\rT(\wtd A).$

{\bf Claim:} $\wtd f$ is continuous on $\rT(\wtd A).$

{\bf Proof of the Claim:} 
Let $\{s_i\tau_i+(1-s_i)\pi\}_{i\in I}\subset \rT(\wtd A)$ be a net 
that converges to 
$\bar \tau\in\rT(\wtd A)$ pointwisely. 
If $\bar\tau=\pi,$ 
then $\{s_i\tau_i\}$ converges to $0$ on $A$ pointwisely.
Since $f\in\Aff(\rT_{[0,1]}(A))$ is continuous, 
we have $\wtd f(\pi)=0
=
\lim_i f(s_i\tau_i) 
=\lim_i\wtd f(s_i\tau_i+(1-s_i)\pi).$ Hence $\wtd f$ is continuous at $\pi.$
%
If $\bar \tau\neq \pi,$ then there are 
$s\in(0,1]$ and $\tau\in\rT(A)$ such that $\bar \tau=s\tau+(1-s)\pi.$  
Since $\pi(a)=0$ for all $a\in A,$
$s_i\tau_i$ converges to $s\tau$ on $A$ pointwisely.
Then 
$\wtd f(\bar \tau)=\wtd f(s\tau+(1-s)\pi)
=f(s\tau)
=
\lim_i f(s_i\tau_i) 
=\lim_i\wtd f(s_i\tau_i+(1-s_i)\pi).$ 
Hence $\wtd f$ is continuous at $\bar\tau.$ 
Thus the Claim holds.

Then $\wtd f\in \Aff(\rT(\wtd A)).$
Moreover, $\wtd f$ is bounded on $\rT(\wtd A).$
By \cite[Corollary 3.10]{BPT}, there is $e\in \wtd A_{sa}$ such that 
$\wtd f(\tau)=\tau(e)$ for all $\tau\in\rT(\wtd A).$
Then $0=\wtd f(\pi)=\pi(e)$ implies that $e\in A_{sa}.$ Hence for all $\tau\in \rT(A),$  $f(\tau)=\wtd f(\tau)=\tau(e).$

If in addition, $A$ is algebraically simple and 
$f(\tau)=\tau(e)>0$ for all $\tau\in \rT(A),$ 
then by \cite[Corollary 6.4]{CP79}, 
there is $\bar e\in A_+$ such that 
$f(\tau)=\tau(e)=\tau(\bar e)$ for all $\tau\in\rT(A).$ 
\end{proof}

\begin{prop}\label{elmy18-02}
Let $A$ be an algebraically simple  
\CA\ with {{$\rT(A)\neq \emptyset$}} 
and let $a\in A_+.$
Let $g\in \Aff(\rT_{[0,1]}(A))$ 
with {{$0<g(\tau)<d_\tau(a)$}} for all {{$\tau\in\rT(A).$}}
Then 
there are $b\in\Her(a)_+^1$ and $n\in\N$ such that $b\le f_{1/n}(a)$  and 
$g(\tau)=\tau(b)$ for all {{$\tau\in \rT(A).$}} 
\end{prop} 

\begin{proof} 
By Proposition \ref{elmy18-01}, 
there is $e\in A_+$ such that 
$g(\tau)=\tau(e)$ for all $\tau\in \rT(A).$
Let $\TT(A)$ be 
the set of densely defined traces (see Definition \ref{Dqtr}),
then $\TT(A)=\R_+\cdot \rT(A)$ (\cite[Lemma 4.6]{CP79}). 
Then we have
\beq\label{elmy19-03}
\text{$g(\tau)=\tau(e)$ for all $\tau\in \TT(A).$}
\eneq
Let $\rT_e:=\{\tau\in \TT(A):\tau(e)=1\}$ 
be a compact set. 
Since $1=\tau(e)=g(\tau)<d_\tau(a)=\lim_n\what{f_{1/n}(a)}(\tau)$ 
for all $\tau\in \rT_e,$ 
and $\what{f_{1/n}(a)}$ is continuous on $\rT_e$ for all $n\in\N,$  
by a Dini-type theorem (e.g., \cite[5.6]{Fu2025}), 
there exists $n_0\in\N$ 
such that 
\beq 
1 <\tau(f_{1/n_0}(a)) \mbox{ for all }\tau\in \rT_e. 
\eneq
By Cuntz--Pedersen's result (see \cite[Lemma 7.3]{CP79}),
$e$ is Cuntz--Pedersen subequivalent to $f_{1/n_0}(a),$
which implies that 
there is $b\in\Her(a)_+$ satisfying  
\beq\label{elap30-1}
\mbox{$b\le f_{1/n_0}(a)$ and  
$\tau(e)=\tau(b)$ for all 
$\tau\in \rT(A).$}
\eneq
Then the proposition follows from 
\eqref{elmy19-03} and \eqref{elap30-1}. 
\end{proof}

\begin{rem}
(1) 
For the most part of this paper, 
using 2-quasitraces makes no essential difference to using traces
because they share similar metric properties. 
The above proposition is among the few instances where traces can carry the proof through due to Cuntz--Pedersen's result, whereas 2-quasitraces fail to do so. 
It seems hard --- if not impossible --- to generalize Cuntz-Pedersen's result to 2-quasitraces.

(2) Above proposition  shows that 
under the exact setting, 
affine functions 
automatically 
have 
a form of 
hereditary surjectivity  
--- but  not the one in 
Definition \ref{HeDe}. 
\end{rem}

\section{Tracial diagonal divisibility}

We introduce here two tracial divisibility properties that are closely related to hereditary denseness. 

\begin{df} 
\label{elmy20-02}
Let $A$ be a \CA\ with $\QT(A)\neq \emptyset.$
Let $c\in\Ped(A\otimes\K)_+^1.$ 
We say $c$ is tracially diagonally divisible 
if for every $n\in\N$ and every $\ep>0,$
there is a c.p.c.~order zero map $\phi:M_n\to\Her(c)$ 
such that 
$\|\phi(1_n)- c\|_2<\ep.$ 

In general, a positive element $e\in\Ped(A\otimes\K)_+\nzero$ 
is said to be tracially diagonally divisible if $e/\|e\|$ is
tracially diagonally divisible. 

We say $A$ is tracially diagonally divisible 
if every element in $\Ped(A\otimes\K)_+^1$ is 
tracially diagonally divisible. 

\end{df}

\begin{rem}
The name  ``tracially diagonally divisible'' justifies itself. 
The following definition of tracial diagonal divisibility that makes use of the Cuntz comparison also makes sense (but we will not explore it in this paper): 
an element $c\in\Ped(A\otimes\K)_+^1$ is said to be 
Cu-tracially diagonally divisible if for all $d\in A_+\nzero$ and all $n\in\N,$
there is a c.p.c.~order zero map $\phi: M_n\to \Her(c)$ such that 
$|c-\phi(1_n)|\lesssim d.$ 

\end{rem}

\begin{df}\label{elmy20-01}
Let $A$ be a \CA\ with $\QT(A)\neq \emptyset.$
Let $e\in\Ped(A\otimes\K)_+^1.$  
We say $e$ is weakly tracially diagonally divisible 
if for 
every $n\in\N,$ and every $\ep>0,$
there is a c.p.c.~order zero map $\phi:M_n\to\Her(e)$ 
such that 
$\sup\{|\tau(\phi(1_n))- \tau(e)|:\tau\in\QTAW\}<\ep.$

In general, a positive element $e\in\Ped(A\otimes\K)_+$ 
is said to be weakly tracially diagonally divisible, if $e/\|e\|$ is
weakly tracially diagonally divisible. 

We say $A$ is weakly tracially diagonally divisible, 
if every element in $\Ped(A\otimes\K)_+^1$ is 
weakly tracially diagonally divisible. 
\end{df}

\begin{rem}
Trivially, 
tracial  diagonal divisibility implies weak tracial diagonal divisibility. 
Later, we will show that both concepts are actually equivalent 
for simple separable exact \CAs\ (Theorem \ref{elmy11-10}). 
\end{rem}

\begin{df}\label{elmy13-05}
(\cite[Definition 8.1]{FLosc})
Let $A$ be a \CA\ with $\QT(A)\neq\emptyset.$
We say $A$ has Property (TM) if for all $a\in\Ped(A\otimes\K)_+^1,$
all $n\in\N,$ all $\ep>0,$ there is a c.p.c.~order zero map $\phi:M_n\to\Her(a)$ such that $\|\phi(1_n)a-a\|_2<\ep.$ 
\end{df}
\begin{rem}
``TM'' stands  for ``Tracially Matricial''.  
The existence of such $\phi$ means that $\Her(a)$ has an evenly divided approximate identity $\phi(1_n)$ (with respect to the 2-norm), 
and that tracially locally $\Her(a)$ looks like $M_n(\Her(e_{1,1}))$ --- hence the term TM. 
\end{rem}

\begin{prop}\label{elmy11-07}
Let $A$ be a 
\CA\ with $\TQT(A)\neq\{0\}.$
If $A$ is tracially diagonally divisible, then $A$ has Property (TM). 
\end{prop}
\begin{proof}
Let $e\in\Ped(A\otimes\K)_+,$  $n\in\N$ and  $\ep>0.$ 
Without loss of generality, we may assume $\|e\|\le 1.$ 
Choose $\theta>0$ such that $\|f_\theta(e)e-e\|_2<\ep/8.$
Choose $\dt>0$ such that for all $x\in \Her(e)_+^1,$ if $\|x-e\|_2<\dt,$ then $\|f_\theta(x)-f_\theta(e)\|_2<\ep/8.$ 
Since $e$ is tracially  diagonally divisible, 
there is a c.p.c.~order zero map $\phi: M_n\to\Her(e)$ such that $\|\phi(1_n)-e\|_2<\dt.$
Then by the choice of $\dt,$
$\|f_\theta(\phi(1_n))-f_\theta(e)\|_2<\ep/8.$
Let $\psi:=f_\theta(\phi):M_n\to C^*(\phi(M_n))\subset \Her(e)$ 
be the functional calculus of $\phi$ with respect to $f_\theta$ 
(see \cite[Corollary 4.2]{WZ09OZ}). 
Then $\psi(1_n)=f_\theta(\phi(1_n)),$ and
\beq
\|\psi(1_n)e-e\|_2^{2/3}
&=&
\|f_\theta(\phi(1_n))e-e\|_2^{2/3}
\\ &\le &  
\|f_\theta(\phi(1_n))e-f_\theta(e)e\|_2^{2/3}+\|f_\theta(e)e-e\|_2^{2/3}
\\ &\le &  
\|f_\theta(\phi(1_n))-f_\theta(e)\|_2^{2/3}+(\ep/8)^{2/3}
<\ep^{2/3}.
\eneq 
Then the proposition follows. 
\end{proof}

\begin{thm}\label{elmy11-05}
Let $A$ be a simple exact \CA\ with $\TQT(A)\neq\{0\}.$
If $A$ is weakly  tracially diagonally divisible, 
then the canonical map $\Gamma$ is hereditary dense.  
\end{thm} 
\begin{proof}
Let $e\in\Ped(A\otimes\K)_+$ and let $\ep>0.$
Let $f\in \Aff_+(\TQT(A))$  with 
$f(\tau)<d_\tau(e)$ for all $\tau\in\TQT(A)\nzero.$ 
By  Proposition \ref{elmy18-02}, 
there is $a\in\Her(e)_+^1$
such that 
$f(\tau)=\tau(a)$ for all 
$\tau\in \rT(\Her(e)).$
Since $A$ is exact, all 2-quasitraces on 
$\Her(e)$ are traces (\cite{Haa}). Hence $\TQT(A)=\R_+\cdot \rT(\Her(e)).$
Then 
\beq\label{elmy10-06}
f(\tau)=\tau(a) 
\quad \text{for all } 
\tau\in \TQT(A).
\eneq

Let $n\in\N$ such that 
$1/n< 
\ep/4.$
Since $A$ is weakly tracially diagonally divisible, 
there is a c.p.c.~order zero map 
$\phi: M_n\to \Her(a)$ such that 
\beq\label{elmy07-07}
|\tau(\phi(1_n))-\tau(a)|<
\ep/4 
\quad \text{for all } 
\tau\in\QTAW.
\eneq

By Proposition \ref{elmy06-01}
and using the notations therein,  
the element 
$$x_0:=\sum_{k=1}^{n-1} \chi_k(\phi(e_{k,k}))\in\Her(\phi(1_n))_+^1\subset\Her(e)_+^1$$
satisfies $\|x_0\|\le 1$ and for all $\tau\in\QTAW,$ 
\beq\label{elmy07-08}
\tau(\phi(1_n))-1/n \le \tau((\phi(1_n)-1/n)_+)\le d_\tau(x_0)\le \tau(\phi(1_n)).
\eneq
Hence $|d_\tau(x_0)- \tau(\phi(1_n))|\le 1/n.$
This, together with \eqref{elmy07-07} and \eqref{elmy10-06}, 
implies that 
\beq\label{elmy10-07}
|d_\tau(x_0)-f(\tau)|<\ep 
\quad \mbox{for all $\tau\in\QTAW.$}
\eneq
Then 
\eqref{elmy10-07}, 
$x_0\in\Her(e)_+,$ 
and Theorem \ref{20251216-1} ($(2)\Rightarrow (1)$)
show that 
$\Gamma$ is hereditary dense. 
\end{proof}

\begin{lem}\label{elmy08-04}
Let $A$ be a $\sigma$-unital algebraically simple \CA\ with $\TQT(A)\neq\emptyset.$ 
Let $B\subset\Ped(A\otimes\K)$ be a $\sigma$-unital 
algebraically simple 
hereditary $C^*$-subalgebra. 
Then 

(1) $\TQT(A)=\TQT(B).$ 

(2)
For all $x\in B,$ let $\|x\|_{2,B}:=\sup\{\tau(x^*x)^{1/2}:\tau\in
\TQT(A),\|\tau|_B\|\le 1\}.$
Then there are $c_0,c_1>0$ such that $c_0\|x\|_2\le \|x\|_{2,B}
\le c_1\|x\|_{2}$ for all $x\in B.$

(3) Let $J_B:=\{\{b_n\}\in l^\infty(B):
\lim_n\|b_n\|_{2,B}=0\}.$  
Then $J_B=\{\{b_n\}\in l^\infty(B):
\lim_n\|b_n\|_{2}=0\}.$  
%
\end{lem}
\begin{proof}
(1) Since $A$ is simple and $\sigma$-unital, 
Brown's stable isomorphism theorem  \cite{Br77} 
ensures that
$A\otimes\K\cong B\otimes\K$ and hence 
$\TQT(A)=\TQT(B).$
%

(2) 
Let $\QT(B):=\{\tau\in\TQT(A)=\TQT(B):\|\tau|_B\|= 1\}.$ 
By \cite[2.9]{FLosc}, 
$\ol{\QT(B)}^w$ is compact and $0\not\in \ol{\QT(B)}^w.$ 
By \cite[Proposition 2.10 (2)]{FLosc}, we have 
$0<C_0:=\sup\{\|\tau|_B\|:\tau\in\ol{\QT(A)}^w\}<\infty$
and 
$C_1:=\sup\{\|\tau|_A\|:\tau\in \ol{\QT(B)}^w\}<\infty.$

Fix $x\in B.$ 
Let $\tau\in \ol{\QT(A)}^w$ and 
let $\tau_0:=\tau/\|\tau|_B\|\in \QT(B).$
Then $\|x\|_{2,\tau}=\|\tau|_B\|^{1/2} \|x\|_{2,\tau_0}
\le C_0^{1/2}\|x\|_{2,B}.$ 
Hence $\|x\|_{2}\le C_0^{1/2}\|x\|_{2,B}.$

Let $\tau\in  \ol{\QT(B)}^w$ and let $\tau_1:=\tau/\|\tau|_A\|\in\QT(A).$
Then $\|x\|_{2,\tau}=\|\tau|_A\|^{1/2}\cdot \|x\|_{2,\tau_1}\le C_1^{1/2}\|x\|_2.$
It follows that $\|x\|_{2,B}\le C_1^{1/2}\|x\|_2.$ 
Then $c_0:=1/C_0^{1/2}$ and $c_1:=C_1^{1/2}$ are the desired constants. 

(3) This is a corollary of (2). 
\end{proof}

\begin{thm}\label{elmy11-04}
Let $A$ be a $\sigma$-unital algebraically simple \CA\ with $\TQT(A)\neq\emptyset.$
If $A$ has tracial approximate oscillation zero, 
then $A$ is tracially diagonally  divisible. 
\end{thm}
\begin{proof}
Let $a\in \Ped(A\otimes\K)_+^1,$
$n\in\N,$ and $\ep>0.$ 
Let $B:=\Her(a).$ 
By Lemma \ref{elmy08-04}, there is $\dt\in (0,\ep/4)$ such that 
if $x\in B_+^1$ satisfying $\|x\|_{2,B}<\dt,$ then $\|x\|_2<\ep/4.$
By 
\cite[Proposition 5.4]{FLosc}, 
$B$ also has tracial approximate oscillation zero. 
Let $C:=l^\infty(B)/J_B$ and let $\pi: l^\infty(B)\to C$ 
be the quotient map. 
By \cite[Theorem 6.4]{FLosc}, 
$C$ has real rank zero. 
Then there are $m\in\N,$ $r_1,...,r_m\in (0,1),$ 
and projections $p_1,...,p_m\in\Her_C(a)_+$ such that 
\beq\label{elmy08-03}
\left\|a-\sum_{i=1}^mr_i p_i\right\|<  \dt/8.
\eneq
Let $B_i:=p_iCp_i,$ $i=1,...,m.$ 
By Corollary \ref{elmy05-cor1}, for each $i=1,...,m,$
there are mutually orthogonal 
projections $q_{i,1},...,q_{i,n}, q_{i,n+1}\in C$ 
satisfy the following:  

(1) $p_i=\sum_{j=1}^{n+1}q_{i,j};$

(2) $q_{i,1},...,q_{i,n}$ are mutually equivalent; 

(3)
$q_{i,n+1}=\pi(\{e_{i,k}\}_{k\in\N})$ 
with $\{e_{i,k}\}_{k\in\N}\in l^\infty(B)_+^1$ and $\|e_{i,k}\|_{2,B}<\dt/8$ for all $k.$ 
\\
By (2), there is an injective *-homomorphism $h_i:M_n\to B_i$ 
such that 
\beq\label{elmy08-01}
h_i(1_n)=\sum_{j=1}^{n}q_{i,j}\le p_i
\quad (i=1,...,m).
\eneq
Since $p_1,...,p_m$ are mutually orthogonal, 
the following $h$ is a c.p.c.~order zero map: 
\beq
h:=\sum_{i=1}^m r_ih_i: M_n\to C, 
\quad
x\mapsto \sum_{i=1}^m r_ih_i(x). 
\eneq
Note that 
\beq\label{elmy08-02}
h(1_n)+\sum_{i=1}^mr_iq_{i,n+1}
\overset{\eqref{elmy08-01}}{=}
\sum_{i=1}^mr_i \left(\sum_{j=1}^{n+1} q_{i,j} \right)
\overset{(1)}{=}
\sum_{i=1}^mr_i p_i.
\eneq
By \cite[Corollary 4.1]{WZ09OZ} and projectivity of $CM_n,$
there is a c.p.c.~order zero map $\phi: M_n\to l^\infty(B)$ such that 
$h=\pi\circ\phi.$ Let $\phi_k:M_n\to B=\Her(a)$ be the components of $\phi.$
For $k\in\N,$ Let 
$
e_k:=\sum_{i=1}^m r_i e_{i,k}. 
$
Then 
\beq\label{elmy13-01}
\pi(\{\phi_k(1_n)+e_k\}_{k\in\N})
\overset{(3)}{=}h(1_n)+\sum_{i=1}^m r_i  q_{i,n+1}
\overset{\eqref{elmy08-02}}{=}
\sum_{i=1}^mr_i p_i. 
\eneq
Then 
\beq
\|a- \pi(\{\phi_k(1_n)+e_k\}_{k\in\N})\|
\overset{\eqref{elmy13-01}}{=}\left\|a- \sum_{i=1}^mr_i p_i\right\|
\overset{\eqref{elmy08-03}}{<}  \dt/8. 
\eneq
Hence there is $\{c_k\}\in (J_B)_+$ such that for all $k\in\N,$
\beq\label{elmy08-08}
\|a- (\phi_k(1_n)+e_k+c_k)\|< \dt/8. 
\eneq
By (3), $\|e_{i,k}\|_{2,B}<\dt/8$ for all $k.$ 
Then \eqref{elmy08-08} and the fact that 
$\lim_{k\to\infty}\|c_k\|_{2,B}=0$
show that  
for sufficiently large $k_0,$ 
$\|a- \phi_{k_0}(1_n)\|_{2,B}<\dt.$ By the choice of $\dt,$ we have
$\|a- \phi_{k_0}(1_n)\|_{2}<\ep,$ thus completing the proof.
\end{proof}

\section{Tracial almost divisibility implies weak tracial diagonal  divisibility}

\begin{df}
{\rm (Winter, \cite[Definition 3.5 (ii)]{W12pure},
see also \cite[Definition 2.7 (iv)]{T14})}
Let $A$ be an algebraically simple \CA\ with $\QT(A)\neq\emptyset.$
Let $m\in\N\cup\{0\}.$
We say $A$ is tracially $m$-almost divisible, 
if for all $m\in\N$ and all $a\in M_m(A)_+^1,$ 
all $n\in\N,$ and all $\ep>0,$
there is a c.p.c.~order zero map $\phi: M_n\to\Her(a)$
such that $\tau(1_n)\ge \frac{1}{m+1}\cdot \tau(a)-\ep$ for all $\tau\in\QT(A).$ 

We say $A$ is tracially almost divisible, if $A$ is tracially $0$-almost divisible. 
\end{df}

The following is Tikuisis' theorem 
with minor modifications. 
The proof that Tikuisis' theorem implies the following theorem is standard. We will omit it.
See also \cite[Lemma 5.11]{W12pure} for the unital case. 

\begin{thm}\label{elmy13}
{\rm (Tikuisis, \cite[Theorem 7.6]{T14})}
Let $A$ be a separable, algebraically simple, 
exact $C^*$-algebra with tracial $m$-almost divisibility
and with $\QT(A)\neq\emptyset.$ 
Let $B\subset A$ 
be a separable $C^*$-subalgebra with 
finite nuclear dimension, 
and let $k\in\N.$ 

Then there exists a sequence of c.p.c.~order zero maps  
$\psi_i: M_k\to A$ satisfying the following: 
Let   $\Psi: M_k\to l^\infty(A)/c_0(A)$ be 
the map induced by $\{\psi_i\}_{i\in\N}.$ Then 
\beq
&
\Psi(M_k)\subset(l^\infty(A)/c_0(A))\cap B',\quad\text{and }
&\\&
\lim_{i\to\infty} \sup\{|\tau(b-\psi_i(1_k)b)|: \tau\in \QT(A)\}=0
\quad\text{ for all }b\in B.&
\eneq
\end{thm}

The following theorem is a direct consequence of 
Tikuisis' theorem. 

\begin{thm} \label{elmy10-09} 
Let $A$ be a separable algebraically simple  exact \CA\ 
with $\QT(A)\neq\emptyset$ and 
with tracial $m$-almost divisibility for some $m\in\N\cup\{0\}.$  
Then $A$ is weakly tracially diagonally divisible. 
\end{thm}
\begin{proof}
Let $e\in\Ped(A\otimes\K)_+^1,$  $k\in\N,$ 
and $\ep>0.$ 
Then $B:=C^*(e)$ has nuclear dimension at most $1.$
Note that $A_0:=\Her(e)$ is algebraically simple (\cite[5.6]{CP79}). 
Since $A$ is algebraically simple, 
$0\not\in \ol{\QT(A)}^w$ (\cite[2.9]{FLosc}). 
Let $N:=\sup\{\|\tau|_{A_0}\|:\tau\in\ol{\QT(A)}^w\}<\infty$
(\cite[2.10 (2)]{FLosc}). 
{{Let $\dt:=\ep /(N+1).$}}
Let $\pi: l^\infty(A_0)\to  l^\infty(A_0)/c_0(A_0)$ be the quotient map. 
By Theorem \ref{elmy13}, 
there is a sequence of  c.p.c.~order zero maps $\psi_i: M_k\to A_0$
($i\in\N$) satisfying the following: Let $\Psi: M_k\to l^\infty(A_0)/c_0(A_0)$ be the map induced by $\{\psi_i\}_{i\in\N}.$ Then 
\beq
&
\Psi(M_k)\subset(l^\infty(A_0)/c_0(A_0))\cap B',\quad\text{and }
&\label{elmy13-02}
\\&
\lim_{i\to\infty} \sup\{|\tau(b- \psi_i(1_k)b)|: \tau\in \QT(A_0)\}=0
\quad\text{ for all }b\in B.&
\label{elmy13-03}
\eneq
Then \eqref{elmy13-02} shows that the map 
$\Phi:M_k\to (l^\infty(A_0)/c_0(A_0))\cap B',$
$x\mapsto \Psi(x)e$ is also c.p.c.~order zero. 
Note that $\Phi(1_k)=\Psi(1_k)e.$
Let $\bar \Phi: M_k\to l^\infty(A_0)$ be a c.p.c.~order zero lifting of $\Phi$ 
and let $\{\phi_i\}_{i\in\N}$ be the components of $\bar\Phi.$
Then $\lim_i\|\phi_i(1_k)- \psi_i(1_k)e\|=0.$ 
This together with \eqref{elmy13-03} show that there is $n_0\in\N$
such that 
\beq
\sup\{|\tau(e)-\tau(\phi_{n_0}(1_k))|:\tau\in\QT(A_0)\}<\dt.
\eneq
For every $\lambda\in \QT(A),$ 
let $\lambda_0:=\lambda/\|\lambda|_{A_0}\|\in\QT(A_0).$
Then 
\beq
|\lambda(e)-\lambda(\phi_{n_0}(1_k))|=\|\lambda|_{A_0}\|\cdot |\lambda_0(e)-\lambda_0(\phi_{n_0}(1_k))|\le N \cdot \dt<\ep.
\eneq
Then the theorem follows. 
\end{proof}

\section{Proof of Theorem \ref{elmy11-10}}
Summarizing all the results, we come to the proof of Theorem \ref{elmy11-10}. 


\ \\
{\bf Proof of Theorem \ref{elmy11-10}: }

$(1)\Rightarrow(2)$ is Corollary \ref{elmy11-03}.

$(2)\Rightarrow(3)$ is \cite[Theorem 6.4]{FLosc} 
and 
$(3)\Rightarrow(2)$ is in \cite[Theorem 3.12]{Fu2025}. 

$(3)\Rightarrow(4)$ is by the equivalence of $(3)$ and $(2)$ and Theorem \ref{elmy11-04}. 

$(4)\Rightarrow(5)$ is Proposition \ref{elmy11-07}. 

$(5)\Rightarrow(6)$ is Proposition \ref{elmy06-07}.

$(6)\Rightarrow(7)$ is trivial.

$(7)\Rightarrow(8)$ (assuming $A$ exact) is Theorem \ref{elmy10-09} (see also \cite[Theorem 7.6]{T14}).

$(8)\Rightarrow(1)$ (assuming $A$ exact) is Theorem \ref{elmy11-05}.  

\qed


The following direct corollary is reminiscent of Antoine-Perera-Robert-Thiel's result that simple $(m,n)$-pure \CAs\ are pure 
(\cite[Theorem D]{APRT-JFA}). 
Here, we need not assume any comparison property for positive elements.

\begin{cor}
Let $A$ be an algebraically simple separable exact \CA\ with $\QT(A)\neq\emptyset.$
Then $A$ is tracially almost divisible if and only if $A$ is  tracially $m$-almost divisible for some $m\in\N\cup\{0\}.$
\qed
\end{cor}

Recall that the Cuntz semigroup $\Cu(A)$ of $A$ is said to be almost divisible, if for all $a\in M_\infty(A)$ and all $k\in\N$ there is $x\in \Cu(A)$ such that $kx\le [a]\le (k+1)(m+1)x$ (\cite[Definition 3.5]{W12pure}). 
See also \cite[Proposition 5.2]{CETW-Gamma} 
for equivalent descriptions of almost divisible when strict comparison is provided.

\begin{cor}
Let $A$ be an algebraically simple separable exact \CA\ with $\QT(A)\neq\emptyset.$
If  $\Cu(A)$ is $m$-almost divisible for some $m\in\N\cup\{0\},$
then $l^\infty(A)/J_A$ has real rank zero. 
\end{cor}
\begin{proof}
By \cite[Proposition 2.9 (v)]{T14}, 
$A$ is tracially $m$-almost divisible
(see \cite[Proposition 3.8]{W12pure} for the unital case).
Then apply Theorem \ref{elmy11-10}. 
\end{proof}

\begin{cor}\label{elmy20-05}
Let $A$ be an algebraically simple 
separable exact stable rank one \CA.
Then $A$ has all the properties listed in Theorem \ref{elmy11-10}. 
\end{cor}
\begin{proof}
By \cite[Theorem A]{Fu2025},
$A$ has tracial approximate oscillation zero. Then apply Theorem \ref{elmy11-10}. 
\end{proof}

\begin{cor} 
The following \CAs\ have all the properties listed in Theorem \ref{elmy11-10}: 
(1) Simple diagonal AH-algebras. 
(2) $C(X)\rtimes \Z^d,$ where  $(X,\Z^d)$ is  a minimal free topological dynamical system.  
(3) $C(X)\rtimes \Gamma,$ where   $\Gamma$ is  a countable discrete amenable group,
and $(X,\Gamma)$ is a minimal free topological 
dynamical system with (URP) and (COS). 
\end{cor}
\begin{proof}
The listed $C^*$-algebras have stable rank one by 
\cite[Theorem 4.1]{EHT}, 
\cite[Corollary 7.9]{LN}, and 
\cite[Theorem 7.8]{LN}. 
To conclude, we apply Corollary \ref{elmy20-05}.
\end{proof}


\section{Applications}

\subsection{
Uniformly McDuff,  uniform property 
\texorpdfstring{$\Gamma$}{Gamma}, and CPoU}

\begin{df}
Let $A$ be a \CA, let $\rT(A)$ be the tracial state space of $A.$
Let $\omega$ be a free ultrafilter on $\N.$ 
Let $\{\tau_n\}_{n\in\N}\subset T(A).$
Define $\tau: l^\infty(A)\to\C$ by $\tau(\{a_n\}):=\lim_{n\to\omega}\tau_n(a_n).$ Such $\tau$ is called a limit trace on $l^\infty(A).$
Let ${\rm T}_\omega(A)$ be the set of all limit traces on $l^\infty(A).$
Let $J_{A,\omega}:=\{\{x_n\}\in l^\infty(A):\lim_\omega \|x_n\|_2=0\}.$
Let $A^\omega:=l^\infty(A)/J_{A,\omega}.$
\end{df}

Note that if $A$ is separable and ${\rm T}(A)$ is non-empty and compact, 
then $A^\omega$ is unital
(\cite[Proposition 1.11]{CETWW}). 

\begin{df}(\cite[Definition 4.2]{CETW-Gamma})
Let $A$ be a  separable \CA\ with ${\rm T}(A)$ non-empty and compact. 
$A$ is said to be uniformly McDuff 
if for all $n\in\N,$ there exists a unital embedding 
$h: M_n\to A^\omega\cap A'.$
\end{df} 
\begin{rem} 
Uniformly McDuff is one of the forms of 
``tracial approximate divisibility''. 
See also \cite[Definition 4.8, Theorem 4.11]{FLL21} for 
several equivalent versions of 
Property (TAD) that use Cuntz comparison, 
where TAD stands for  ``Tracial Approximate Divisibility''. 
Uniformly McDuff is exactly Property (TAD-3) when strict comparison is provided. 
\end{rem}

\begin{df}
(\cite{CETWW})
Let $A$ be a  separable \CA\ with ${\rm T}(A)$ non-empty and compact. 
$A$ is said to have uniform property $\Gamma$ 
if for all $n\in\N,$ there exist projections  
$p_1,...,p_n\in 
A^\omega \cap A'$ 
summing to $1_{A^\omega}$
such that 
for all $a\in A,$ all $i=1,...,n,$ and all $\tau\in {\rm T}_\omega(A),$
$
\tau(p_ia)= \tau(a)/n.
$

\end{df}

\begin{df}
\cite[Definition 3.1]{CETWW}
Let $A$ be a  separable \CA\ with ${\rm T}(A)$ non-empty and compact. 
$A$ is said to have 
complemented partitions of unity (CPoU) if for
every $\|\cdot\|_{2,{\rm T}_\omega(A)}$-separable subset 
$S$ of $A^\omega,$ 
every family $a_1,...,a_k\in (A^\omega)_+,$
and any scalar 
$
\dt>\sup_{\tau\in {\rm T}_\omega(A)}\min\{\tau(a_1),...,\tau(a_k)\},
$
there exist orthogonal projections 
$p_1,...,p_k\in A^\omega\cap S'$ such that 
$\sum_{i=1}^k p_i=1_{A^\omega}$ and 
$
\tau(a_ip_i)\le \dt\tau(p_i)$ 
for all $\tau\in {\rm T}_\omega(A)$ 
$(i=1,...,k).$

\end{df}

The following theorem has different proofs 
(but all need to use Theorem \ref{elmy11-10}); 
see Remark \ref{elmy19-05}. 
The following proof requires no extra tools.

\begin{thm}
Let $A$ be an algebraically simple separable nuclear \CA\ with ${\rm T}(A)$ non-empty and compact. 
Assume that one of the following conditions hold: 

(1) $A$ is uniformly McDuff.

(2) $A$ has uniform property $\Gamma.$

(3) $A$ has CPoU. 
\\
Then all the eight properties stated in Theorem \ref{elmy11-10} hold for $A.$ 

\end{thm} 

\begin{proof} 
(1) Assume that $A$ is uniformly McDuff. 
Let $n,m\in\N,$ $a\in M_m(A)_+^1,$ 
and $\ep>0.$
Let $B:=M_m(A).$ 
Since central sequence algebra is stable under taking matrices, 
$B$ is also uniformly McDuff. 
Then there is a unital embedding 
$\phi: M_n\to B^\omega\cap B'.$ 
Let $\pi: l^\infty(B)\to B^\omega$ be the quotient map. 
Since $\phi(M_n)$ commutes with $a,$ 
the map $\psi: M_n\to \pi(l^\infty(\Her_B(a))),$
$x\mapsto \phi(x)a$ is c.p.c.~order zero
with 
\beq\label{elmy15-05}
\psi(1_n)=a.
\eneq
Let $\bar \psi:M_n\to l^\infty(\Her_B(a))$ 
be a c.p.c.~order zero lifting of $\psi$
and let $\psi_i: M_n\to \Her_B(a)$ $(i\in\N)$ 
be the components of $\bar\psi.$ 
Then \eqref{elmy15-05} shows that 
there is $i_0\in\N$ such that $\|\psi_{i_0}(1_n)-a\|_{2}<\ep.$
Hence $A$ is tracially diagonally divisible,
and thus the eight properties in  Theorem \ref{elmy11-10} 
hold for $A.$

Note that 
(1), (2), and (3) are equivalent for the given $A$
by \cite[Theorem 4.6]{CETW-Gamma}. 
It follows that 
the rest of the theorem also holds.
\end{proof}

\begin{rem}\label{elmy19-05} 
We note that there are different approaches to 
proving the above theorem. 
For instance, \cite[Proposition 7.2]{TraCom} 
can be used to deduce the above theorem. 
Indeed, we have  $A^\omega=(\ol{A}^{T(A)})^\omega,$ where 
$\ol{A}^{T(A)}$ is the uniform tracial completion of $A$
and is factorial tracially complete. 
Applying \cite[Proposition 7.2]{TraCom}, 
we obtain that $A^\omega=(\ol{A}^{T(A)})^\omega$ has real rank zero.
Then by Theorem \ref{elmy11-10}, 
$A$ also satisfies all the eight properties stated therein.

Recently, Elliott-Niu (\cite{EN}) proved that 
for a unital \CA\ $B$, 
uniform property $\Gamma$ implies Property (S),
where Property (S) is defined for unital \CAs, 
and $B$ having Property (S) is equivalent to $l^\infty(B)/J_B$ having  real rank zero. 
Then it follows from Theorem \ref{elmy11-10} that 
$B$ also satisfies all the eight properties listed there.

\end{rem}



\begin{df}
{\rm (\cite[Definition 4.1 (iv)]{W12pure})}
Let $A$ be a \CA. $A$ is said to have locally finite nuclear dimension, 
if for all finite subset $F,$ all $\ep>0,$
there is a $C^*$-subalgebra $B\subset A$ such that
$B$ has finite nuclear dimension and $F\subset_\ep B.$
\end{df}
The following result removes the unital condition in  Vaccaro's result (\cite{Vacc}).  


\begin{thm}\label{elmy14-06}
{\rm (cf.~\cite[Theorem B]{Vacc})}
Let $A$ be an algebraically simple separable 
non-elementary \CA\ with ${\rm T}(A)$ non-empty and compact
and with stable rank one. 
If $A$ has locally finite nuclear dimension, 
then the following hold: 

(1) $A$ is uniformly McDuff. 

(2) $A$ has uniform property $\Gamma.$

(3) $A$ has CPoU. 

\end{thm}
\begin{proof} 
(1): Let $\pi: l^\infty(A)\to l^\infty(A)/J_A$ be the quotient map. 
Let $e\in A_+^1$ be a strictly positive element of $A.$ 
Define $h_n: {\rm T}(A)\to [0,1],$  $\tau\mapsto \tau(f_{1/n}(e))$
be an increasing sequence of continuous functions on ${\rm T}(A).$
Since ${\rm T}(A)$ is compact
and $\lim_n\tau(f_{1/n}(e))=1$ for all $\tau\in {\rm T}(A),$
by Dini's theorem, 
$\{h_n\}$ converges to $1$ on ${\rm T}(A)$ uniformly. 
It follows that $\pi(\{f_{1/n}(e)\}_{n\in\N})$ is the unit of 
$(l^\infty(A)/J_A)\cap A'$ (see \cite[Lemma 1.10, Proposition 1.11]{CETWW}). 

Fix $k\in\N.$ 
Let $F_1\subset F_2\subset..$ 
be a an increasing sequence of finite subsets in the unit ball of $A_+^1$ such that $\cup_{m\in\N} F_m$ is dense in $A_+^1.$

Since $A$ has locally finite nuclear dimension, 
for each $n\in \N,$ 
there is a finite nuclear dimension $C^*$-subalgebra $B_n\subset A$ 
and a finite subset $Y_n\subset B^1_{n,+}$ 
such that  $F_n\cup\{f_{1/n}(e)\}\subset_{1/n^2} Y_n.$ 
In particular, let $y_n\in Y_n$ such that 
\beq\label{elmy14-01}
\|f_{1/n}(e)-y_n\|\le 1/n^2. 
\eneq
By  Theorem \ref{elmy13},
for each $n\in\N,$ 
there is a c.p.c.~order zero map 
$\phi_n: M_k\to A$ such that 

\quad (i) $\|\phi_n(x)y-y\phi_n(x)\|<1/n^2$ for all $x\in M_k^1$ and all $y\in Y_n,$ and

\quad (ii) $\sup\{|\tau(y_n-\phi_n(1_k)y_n)|:\tau\in {\rm T}(A)\}<1/n^2.$
\ \\
By (ii), we have 
\beq\label{elmy14-03}
\|y_n-\phi_n(1_k)y_n\|_2\le 
\sup\{|\tau(y_n-\phi_n(1_k)y_n)|:\tau\in {\rm T}(A)\}^{1/2}<1/n. 
\eneq
Let $\phi: M_k\to l^\infty(A)/J_A$ be the map induced by $\{\phi_n\}_{n\in\N}.$ 
By \eqref{elmy14-01} and  \eqref{elmy14-03}, 
\beq
\phi(1_k)\cdot \pi(\{f_{1/n}(e)\}
=\pi(\{\phi_n(1_k) f_{1/n}(e)\})
=\pi(\{\phi_n(1_k)y_n\})=\pi(\{y_n\})=\pi(\{f_{1/n}(e)\}).
\nonumber\eneq
Then since $\pi(\{f_{1/n}(e)\})$ is the unit of $(l^\infty(A)/J_A)\cap A',$
$\phi(1_k)$ is as well.
Moreover, by (i), 
$\phi(M_k)
\subset 
(l^\infty(A)/J_A)\cap A'.
$ 
Hence $A$ is uniformly McDuff. 

(2) and (3): Since $A$ is uniformly McDuff, 
it follows from 
\cite[Theorem 4.6]{CETW-Gamma}
that $A$ has uniform property $\Gamma$ and CPoU. 
\end{proof}

\subsection{Tracial completion} 


Tracially complete \CA\ is a powerful tool that 
has been systematically  investigated in \cite{TraCom}.
We give here an application to tracial completion. 
The next theorem requires a number of concepts and definitions that are beyond the scope of the present paper to list; 
we therefore content ourselves with providing 
the reader with precise references as follows.


1. Tracially complete \CA\ 
(\cite[Definition 3.4]{TraCom}). 

2. Type ${\rm II_1}$ tracially complete \CA\ 
(\cite[Definition 3.8]{TraCom}).

3. Factorial tracially complete \CA\ 
(\cite[Definition 3.13]{TraCom}).

4. Tracial completion (\cite[Definition 3.19]{TraCom}).

5. The hyperfinite model $({\cal R}_{X},X)$ for the classifiable tracially complete $C^*$-algebras corresponding to 
a given metrizable Choquet simplex $X$ 
(\cite[Example 3.35]{TraCom}). 

6. 
Amenable tracially complete \CA\ 
(\cite[Definition 4.1]{TraCom}). 

7. 
McDuff tracially complete \CA\ 
(\cite[Definition 5.12]{TraCom}). 

8.
Property $\Gamma$ for factorial tracially complete \CA\ 
(\cite[Definition 5.19]{TraCom}). 

9.
CPoU for factorial tracially complete \CA\ 
(\cite[Definition 6.1]{TraCom}). 

10. 
Hyperfinite type ${\rm II}_1$ tracially complete \CA\ 
(\cite[Definition 8.1]{TraCom}).

11. Pure $C^*$-algebra (\cite[Definition 3.6]{W12pure}).


\begin{df}
(\cite[3.19]{TraCom})
Let $A$ be a \CA\ with ${\rm T}(A)$ non-empty. 
The uniform tracial
completion of $A$ with respect to $\rT(A)$ is 
the following 
$C^*$-subalgebra of $l^\infty(A)/J_{A}:$ 
\beq
\ol{A}^{\rT(A)}:=
\{\{a_n\}_{n\in\N}\in l^\infty(A): \{a_n\}_{n\in\N}\text{ is $\|\cdot\|_{2}$-Cauchy}\}/J_A.
\eneq 
\end{df} 

\begin{nota}
Let $A$ be a \CA\ with ${\rm T}(A)$ non-empty and compact.
Let $\wtd{X}$ be the set of all traces on $\ol A^{\rT(A)}$
that induced by traces in $\rT(A).$ 
By \cite[Proposition 3.23 (iii), (iv)]{TraCom}, 
$\left(\ol A^{\rT(A)}, \wtd X\right)$ is 
a factorial  tracially complete \CA. 
By \cite[Proposition 3.23 (ii)]{TraCom}, $\rT(A)$ is affinely homeomorphic to $\wtd X.$ 
Hence, following the convention of \cite[Notation 3.24]{TraCom}, 
we will use the notation $\left(\ol A^{\rT(A)}, \rT(A)\right)$ 
hereafter 
to refer to the  tracial completion of $A$ with respect to $\rT(A).$
\end{nota}


\begin{thm} \label{elmy14-08}
Let $A$ be an algebraically simple separable 
non-elementary 
stable rank one \CA\ with ${\rm T}(A)$ non-empty and compact. 
If $A$ has  locally finite nuclear dimension,
then 
the tracial completion 
$\left(\ol A^{\rT(A)}, \rT(A)\right)$ 
is a hyperfinite type ${\rm II_1}$ 
factorial tracially complete $C^*$-algebra
that isomorphic to the model $({\cal R}_{\rT(A)},\rT(A)).$ 
Moreover, 
$\left(\ol A^{\rT(A)}, \rT(A)\right)$ 
is amenable,  McDuff, 
has property $\Gamma$ 
and CPoU. 
Furthermore, 
$\ol{A}^{{\rm T}(A)}$ 
is pure, 
and has real rank zero and stable rank one,
and satisfies $\rT\left(\ol A^{\rT(A)}\right)= \rT(A).$

\end{thm}
\begin{proof}
{{By \cite[Corollary 4.10]{TraCom}, 
$\left(\ol A^{\rT(A)}, \rT(A)\right)$ 
is an amenable factorial tracially complete \CA.}} 
By Theorem \ref{elmy14-06}, 
$A$ is uniformly McDuff. 
By \cite[Proposition 5.13]{TraCom}, 
{{$\left(\ol A^{\rT(A)}, \rT(A)\right)$}} is  McDuff. 
Then for all $\tau\in\rT(A),$ 
the weak closure $\pi_\tau\left(\ol A^{\rT(A)}\right)''$  of 
$\ol A^{\rT(A)}$ in the GNS representation with respect to $\tau$ 
has no type ${\rm I}$ part. 
Hence by \cite[Definition 3.8]{TraCom},  $\left(\ol A^{\rT(A)}, \rT(A)\right)$ is type ${\rm II_1}.$
%
By \cite[Theorem 9.15 (iv)]{TraCom}, 
$\left(\ol A^{\rT(A)}, \rT(A)\right)$ 
is hyperfinite, 
has property $\Gamma$ 
and CPoU.
{{
Since $A$ is $\|\cdot\|$-separable, 
$\ol A^{\rT(A)}$ is $\|\cdot\|_{2,\rT(A)}$-separable. 
By \cite[Theorem 9.15]{TraCom}, 
$\left(\ol A^{\rT(A)}, \rT(A)\right)$ 
is  isomorphic to the model $({\cal R}_{\rT(A)},\rT(A))$ 
that constructed in \cite[Example 3.35]{TraCom}.}} 
The rest of the theorem is a direct application of 
Evington-Tikuisis' recent results \cite[Theorem A, B, C, D]{ET26}.
\end{proof}

\begin{cor}\label{elmy15-03}
Let $A$ 
be an algebraically simple separable  stable rank one 
\CA\ with ${\rm T}(A)$ non-empty and compact
and with locally finite nuclear dimension. 
Then $A$ has the following form of tracial strict comparison: 
For every $a,b\in A_+,$ if $d_\tau(a)<d_\tau(b)$ for all $\tau\in\rT(A),$
then there is a sequence $\{r_n\}\subset A$ such that 
$\lim_n\|a-r_n^*br_n\|_2=0.$
\end{cor} 
\begin{proof}

let $a,b\in A_+$ such that 
$d_\tau(a)<d_\tau(b)$ for all $\tau\in\rT(A).$
Then there is $k\in\N$ such that $(k+1)[a]\le k[b]$ in $\Cu(A).$
Hence $(k+1)[a]\le k[b]$ in $\Cu(\ol A^{\rT(A)}).$
Since $\ol A^{\rT(A)}$ is pure (Theorem \ref{elmy14-08}), 
we have $a\lesssim b$ in $\ol A^{\rT(A)}.$ 
Then there is a sequence $\{s_n\}\subset \ol A^{\rT(A)}$ such that 
$\lim_n\|a-s_n^*bs_n\|=0.$
Since $A$ is $\|\cdot\|_{2}$-dense in $\ol A^{\rT(A)},$ 
we can approximate each $s_n$ by some $r_n\in A$ (in 2-norm)
such that $\|s_n^*bs_n-r_n^*br_n\|_2<1/n.$ 
Then we have 
$\lim_n\|a-r_n^*br_n\|_2=0.$ 
\end{proof}

In \cite{V97} 
J.~Villadsen constructed  simple stable rank one AH-algebras that 
fail to have strict comparison. 
We refer to such constructions as Villadsen algebras of the first type. 
Above corollary asserts that Villadsen algebras of the first type
have tracial strict comparison. 


\appendix 
\renewcommand{\theequation}{e\ \thesection.\arabic{equation}}  
\setcounter{equation}{0}

\section{Appendix} 

This Appendix is basically independent of the other parts of this paper. 

Firstly, we record here two known facts from \cite{FLosc} with a bit of explanation. 

The following lemma, which  is essentially a combination of 
\cite[Lemma 7.4]{FLosc} and \cite[Lemma 7.6]{FLosc}, 
is reminiscent of S.~Zhang's result on approximately halving projections in simple real rank zero algebras. 

\begin{lem}\label{elmy05-lem1}
Let $A$ be a simple non-elementary \CA\ with $\QT(A)\neq \emptyset.$
Assume that $A$ has tracial approximate oscillation zero. 
Let $p\in l^\infty(A)/J_A$ be a projection. 
Let $n\in\N$ and let $\dt>0.$ 

(1) There is a full projection $e=\pi(\{e_i\})\in l^\infty(A)/J_A$
with $\{e_i\}\subset l^\infty(A)_+$ and $\sup\{d_\tau(e_i):\tau\in\QT(A)\}<\dt.$

(2) There are mutually orthogonal 
projections $p_1,...,p_{2^n}, p_{2^n+1}\in l^\infty(A)/J_A$ 
such that 
$p=\sum_{i=1}^{2^n+1}p_i$
with $p_1,...,p_{2^n}$ being mutually equivalent, 
and $p_{2^n+1}\lesssim e.$ 
\end{lem}
\begin{proof}
By \cite[Lemma 7.4]{FLosc}, the following (i)-(iii) hold: 

(i) There is a full projection $\bar e=\pi(\{\bar e_i\})\in l^\infty(A)/J_A$
with $\{e_i\}\subset l^\infty(A)_+$ and $\sup\{d_\tau(e_i):\tau\in\QT(A)\}<\dt/4.$ 


(ii) There is a sequence of mutually orthogonal full projections 
$r_1,r_2,...\in l^\infty(A)/J_A$ 
such that $er_i=0$ for all $i\in\N,$ and $2^{2k}[r_k]\le[\bar e].$

(iii) For each $k,$  there are mutually orthogonal and mutually equivalent full projections $r_{k,1},...r_{k,2^{k+1}}\in \Her(r_k).$

Note that, if $p\lesssim \bar e,$ then we can take $p_1=...=p_{2^n}=0,$ $p_{2^n+1}:=p,$ 
$e:=\bar e,$ $e_i:=\bar e_i$ for all $i\in\N.$ Then both (1) and (2) hold. 
Hence in the following we may assume that 
$p\not\lesssim \bar e.$

Note that (i), (ii), and (iii) fit into the assumptions of \cite[Lemma 7.6]{FLosc}. 
Then by \cite[Lemma 7.6]{FLosc}, 
there are mutually orthogonal  
projections $p_1,...,p_{2^n}, p_{2^n+1}\in l^\infty(A)/J_A$ 
such that $p_1,...,p_{2^n}$ are mutually equivalent, 
and $p_{2^n+1}\lesssim \bar e+ e',$ where $e'$ is a projection, which can be written as a finite sum of $r_{k,j},$ and moreover, $e'\lesssim \bar e.$

Now define $e:=\bar e+e'.$ Then (2) of the lemma holds. 

Since $e'\lesssim \bar e,$ there is a sequence $\{v_i\}\in l^\infty(A)^1$ such that $e'=\pi(\{v_iv_i^*\})$
and $v_i^*v_i \in \Her(\bar e_i)$ for all $i\in\N.$
Define $e_i=\bar e_i+v_iv_i^*.$ Then $e=\pi(\{e_i\}).$
For all $i\in\N$ and all $\tau\in\QT(A),$ 
$d_\tau(e_i)\le d_\tau(\bar e_i)+d_\tau(v_iv_i^*)\le 2 d_\tau(\bar e_i)<\dt/2.$ Hence (1) holds. 
\end{proof}

\begin{cor}\label{elmy05-cor1}
Let $A$ be a simple non-elementary \CA\ with $\QT(A)\neq \emptyset.$
Assume that $A$ has tracial approximate oscillation zero. 
Let $p\in l^\infty(A)/J_A$ be a projection. 
Let $n\in\N$ and let $\dt>0.$ 
Then there are mutually orthogonal 
projections $p_1,...,p_{n}, p_{n+1}\in l^\infty(A)/J_A$ 
satisfy the following:  
(1) $p=\sum_{i=1}^{n+1}p_i;$
(2) $p_1,...,p_{n}$ are mutually equivalent; 
(3)
$p_{n+1}=\pi(\{e_i\})\in l^\infty(A)/J_A$
with $\{e_i\}\in l^\infty(A)_+^1$ 
such that $\sup\{d_\tau(e_i):\tau\in\QT(A)\}<\dt$ 
and $\|e_i\|_2<\dt$
for all $i\in\N.$ 
\end{cor} 
\begin{proof}
This is a direct corollary of Lemma \ref{elmy05-lem1}.
\end{proof}


Secondly, we give a direct proof of the following proposition 
using Corollary \ref{elmy05-cor1}
(see Definition \ref{elmy13-04} for tracial approximate oscillation zero and Definition \ref{elmy13-05} for Property (TM)): 

\begin{prop}\label{elmy05-01}
Let $A$ be a simple non-elementary \CA\ with $\QT(A)\neq \emptyset.$
Assume that $A$ has tracial approximate oscillation zero, 
then $A$ has Property (TM). 
\end{prop}
\begin{proof}
Let $B:=l^\infty(A)/J_A$ and let $\pi: l^\infty(A)\to B$ be the quotient map. 
Let $a\in \Ped(A\otimes\K)_+^1\setminus\{0\},$
$\ep>0,$ and $n\in\N.$ 
Let $S:=\ol{\QT(A)}^w.$ 
Let $\dt<\ep/8.$
Since $A$ has tracial approximate oscillation zero, 
there is a sequence $\{a_i\}_i\subset\Her(a)_+^1$ such that 
\beq\label{elmy06-05}
\lim_{i\to\infty}\|a-a_i\|_2=0 
\mbox{\quad and \quad}
\lim_{i\to\infty} \omega(a_i)=0.
\eneq
Since $\lim_{n\to\infty} \omega(a_i)=0,$ there is a sequence $\{\dt_i\}\subset\R_+$ such that $\lim_i\dt_i=0$ and 
\beq
\lim_i \sup\{\tau(f_{\dt_i}(a_i)-f_{2\dt_i}(a_i)):\tau\in S\}=0. 
\eneq
Then 
$
0\le \lim_i \sup_{\tau\in S}\{\tau(f_{\dt_i}(a_i)-f_{\dt_i}(a_i)^2)\}
\le \lim_i \sup_{\tau\in S}\{\tau(f_{\dt_i}(a_i)-f_{2\dt_i}(a_i))\}=0. 
$
Hence 
$
\lim_i \|f_{\dt_i}(a_i)-f_{\dt_i}(a_i)^2\|_2=0.
$
Thus
$
\pi(\{f_{\dt_i}(a_i)\})-\pi(\{f_{\dt_i}(a_i)\})^2=0.
$
Therefore  $p:=\pi(\{f_{\dt_i}(a_i)\})\in\pi(l^\infty(\Her(a)))$ is a projection. 
Since $\lim_i \|f_{\dt_i}(a_i) a_i-a_i\|=0,$ by \eqref{elmy06-05}, 
we have 
\beq\label{elmy06-06}
pa=a.
\eneq

By Corollary \ref{elmy05-cor1}, 
there are mutually orthogonal projections $p_1,...p_{n}, p_{n+1}$ satisfy the following: 

(1) $p=\sum_{i=1}^{n+1}p_i;$

(2) $p_1,...,p_{n}$ are mutually equivalent; 

(3)
$p_{n+1}=\pi(\{e_i\})\in l^\infty(A)/J_A$
with $\{e_i\}\in l^\infty(A)_+^1$ and 
$\|e_i\|_2<\dt$
for all $i\in\N.$

Since $p_1,...p_{n}\in \pi(l^\infty(\Her(a)))$ are mutually orthogonal and mutually equivalent projections,
there is an 
injective *-homomorphism 
$\bar \phi: M_{n}\to 
\pi(l^\infty(\Her(a)))$ such that 
$\bar\phi(1_n)=\sum_{i=1}^np_i=p.$
By the projectivity of the cone over $M_{n}$ (see \cite[Theorem 10.2.1]{Lor}), 
there is a c.p.c.~order zero map $\phi:M_{n}\to l^\infty(\Her(a))$ such that $\bar\phi=\pi\circ\phi.$ Let $\phi_i:M_{n}\to \Her(a)$ be the components of $\phi$ ($i\in\N$). 
By (3), $p=\pi(\{\phi_i(1_n)\})+\pi(\{e_i\}).$ 
By \eqref{elmy06-06}, 
we have 
$pa=a,$ and hence 
\beq
\lim_i\|(\phi_i(1_n)+e_i)a-a\|_2=0.
\eneq
Then there is $m\in\N$ such that 
$\|(\phi_m(1_{n})+e_m)a-a\|_2<\dt.$
By \cite[Lemma 3.5]{Haa} and (3),  
\beq
\|\phi_m(1_{n})a-a\|_2^{2/3}
&\le &
\|(\phi_m(1_{n})+e_m)a-a\|_2^{2/3}+
\|e_ma\|_2^{2/3}
\\ &\le & 
\dt^{2/3}+
\|e_m\|_2^{2/3}
\le  2\dt^{2/3}<\ep^{2/3}.
\eneq
Then the proposition holds. 
\end{proof}


Thirdly, we note a 
fact that follows from the definitions.

\begin{prop}\label{elmy06-07}
Let $A$ be a simple $\sigma$-unital 
non-elementary \CA\ with $\QT(A)\neq \emptyset.$
Assume that $A$  has Property (TM), then $A$ is tracially almost divisible. 
\end{prop} 
\begin{proof}
Let $a\in\Ped(A\otimes\K)_+^1,$ $n\in\N,$ and $\ep>0.$
Note that there is 
$\dt>0$ such that if $x,y\in\Her(a)_+^1$ satisfy $\|x-y\|_2\le \dt,$ then $|\tau(x)-\tau(y)|<\ep/4$ (see, for example, Lemma \ref{elmy07-01}). 
By Property (TM), there is a c.p.c.~order zero map $\phi: M_n\to \Her(a)$ such that $\|a-\phi(1_n)a\|_2<\dt/4.$
Then 
$$
\|a-\phi(1_n)a\phi(1_n)\|_2^{2/3}\le 
\|a-\phi(1_n)a\|_2^{2/3}+\|\phi(1_n)(a- a\phi(1_n))\|_2^{2/3}
\le 2(\dt/4)^{2/3}. 
$$
Hence $\|a-\phi(1_n)a\phi(1_n)\|_2\le \dt.$
By the choice of $\dt,$ 
we have $|\tau(a)-\tau(\phi(1_n)a\phi(1_n))|<\ep.$
Then
$
\tau(\phi(1_n))\ge \tau(\phi(1_n)a\phi(1_n))
\ge \tau(a)-\ep. 
$
Thus $A$ is tracially almost divisible. 
\end{proof}

Finally, a corollary. 
The unital case of the following proposition also appears in Vaccaro's recent work \cite[Theorem B]{Vacc}, which builds on our previous result \cite[Theorem A]{Fu2025}. Here we prove the full proposition in the general case.

\begin{prop} 
{\rm (cf.~\cite[Theorem B]{Vacc})}
Let $A$ be an {{algebraically}} simple separable \CA\ with stable rank one. 
Then $A$ is tracially almost divisible. 
\end{prop}
\begin{proof}
By our previous result (\cite[Theorem A]{Fu2025}), 
$A$ has tracial approximate oscillation zero.
By Proposition \ref{elmy05-cor1} and Proposition \ref{elmy06-07},
$A$ is tracially almost divisible.
\end{proof}

\textsc{Xuanlong Fu}

Key Laboratory of Intelligent Computing and Applications 
(Tongji University), 
Ministry of Education,
School of Mathematical Sciences, 
Tongji University, 
1239 Siping Road, 
Yangpu District, 
Shanghai, China, 200092. 

E-mail: xuanlongfu@tongji.edu.cn.


\begin{thebibliography}{99}
\addcontentsline{toc}{section}{Reference}

\bibitem{AGJP}
M. Amini, N. Golestani, S. Jamali, and N. C. Phillips, 
{\itshape Non-unital tracially ${\cal Z}$-absorbing $C^*$-algebras.}

\bibitem 
{APRT}
R.~Antoine,  F.~Perera,  L.~Robert and H.~Thiel, 
{\itshape $C^*$-algebras of stable rank one and their Cuntz semigroups.} 
Duke Math. J. {\bf 171} (2022), no.~1, 33--99. 

\bibitem{APRT-JFA}
R.~Antoine,  F.~Perera,  L.~Robert and H.~Thiel, 
{\itshape Traces on ultrapowers of $C^*$-algebras.}
J. Funct. Anal. {\bf 286} (2024), no.~8, Article 110341, 65~pp.




\bibitem 
{BH}B.~Blackadar and D. Handelman,
{\itshape Dimension functions and traces on $C^*$-algebra.}
J. Funct. Anal. {\bf 45} (1982), 297--340.

\bibitem{BKR}
B.~Blackadar, A.~Kumjian, and M.~R{\o}rdam, 
{\itshape Approximately central matrix units and the structure of noncommutative tori.} 
K-Theory {\bf 6} (1992) 267--284.



\bibitem{Br77}
L.~G.~Brown,  
{\itshape Stable isomorphism of hereditary subalgebras of  
$C^*$-algebras.}
Pacific J. Math. {\bf 71} (1977), no. 2, 335--348. 

\bibitem{BP91}
L.~G.~Brown, G.~K.~Pedersen, 
{\itshape $C^*$-algebras of real rank zero.} 
J. Funct. Anal. {\bf 99}   (1991), 131--149.

\bibitem{BPT}
N.~P.~Brown, F.~Perera, and A.~Toms,
{\itshape The Cuntz semigroup, the Elliott conjecture, and dimension
functions on $C^*$-algebras.} 
J. Reine Angew. Math. {\bf 621} (2008), 191--211. 


\bibitem{TraCom}
J.~Carrión, J.~Castillejos, S.~Evington, J.~Gabe, 
C.~Schafhauser, A.~Tikuisis, S.~White, 
{\itshape Tracially complete $C^*$-algebras.}
preprint, arXiv:2310.20594v4. 


\bibitem{CETW-Gamma}
J.~Castillejos1, S.~Evington, A.~Tikuisis,  and
S.~White,
{\itshape Uniform Property $\Gamma.$}
Int. Math. Res. Not., Vol. 2022, No. {\bf 13}, pp. 9864--9908. 


\bibitem{CETWW}
J.~Castillejos1, S.~Evington, A.~Tikuisis,  
S.~White, and W.~Winter,
{\itshape Nuclear dimension of simple $C^*$-algebras.}
Invent. Math. {\bf 224} (2021), no. 1, 245--290. 

\bibitem{CLS}
J.~Castillejos, K.~Li, G.~Szab{\'o},
{\itshape On tracial $\mathcal{Z}$-stability of simple non-unital $\mathrm{C}^*$-algebras.}
Canad. J. Math. \textbf{76} (2024), no.~4, 1285--1303.


\bibitem{CP79}
J.~Cuntz and G.~K.~Pedersen, 
{\itshape Equivalence and traces on $C^*$-algebras.}
J.~Funct.~Anal. {\bf 33} (1979), no. 2, 135--164.

 
\bibitem{DT}
M.~Dadarlat and A.~Toms,
{\itshape Ranks of operators in simple $C^*$-algebras.} 
J. Funct. Anal. {\bf 259} (2010), 1209--1229.
 


\bibitem{eglnkk0} 
G.~Elliott, G.~Gong, H.~Lin and Z.~Niu, 
{\itshape The classification of simple separable KK-contractible \CA s with finite nuclear dimension.}  
J. Geom. Phys. {\bf 158} (2020), 103861, 51 pp.

\bibitem{EHT}
G.~Elliott, T.~Ho, and  A.~Toms,  
{\itshape A class of simple  $C^*$-algebras with stable rank one.} 
J.~Funct.~Anal. {\bf 256} (2009), no. 2, 307--322. 

\bibitem{EN} G.~Elliott and Z.~Niu,    
{\itshape On the small boundary property and Z-absorption, II.}   preprint, arXiv: 2504.03611v1. 


\bibitem 
{ERS}
G.~Elliott, L.~Robert, and L.~Santiago, 
{\itshape The cone of lower  semicontinuous traces on a  $C^*$-algebra.} Amer. J. Math \textbf{133} (2011),   969--1005.

\bibitem{ET26}
S.~Evington, A.~Tikuisis,
{\itshape The real and stable rank of tracially complete $C^*$-algebras,} 
preprint,  arXiv:2604.24206. 

\bibitem{Fu2025}
X.~Fu, 
{\itshape From stable rank one to real rank zero: a note on tracial approximate oscillation zero.}
preprint, arXiv: 2512.23911. 

\bibitem{FLL21}
X.~Fu, K.~Li, and H.~Lin, 
{\itshape Tracial approximate divisibility and stable rank one.}  
J. London Math. Soc. {\bf 106} (2022), 3008--3042.


\bibitem{FL2020}
X.~Fu and H.~Lin,  
{\itshape  Tracial approximation in simple $C^*$-algebras.}
 Canadian Journal of Mathematics, Volume {\bf 74} , Issue 4 , August 2022 , pp. 942--1004. 


\bibitem{FL2022}
X.~Fu and H.~Lin,   
{\itshape Nonamenable simple $C^*$-algebras with tracial approximation.} (English summary) Forum Math. Sigma {\bf 10} (2022), Paper No. e14, 50 pp. 

\bibitem 
{FLosc}
X.~Fu and H.~Lin,   
{\itshape Tracial oscillation zero and stable rank one.}
Canad. J. Math. {\bf 77} (2025), no. 2, 563--630.

\bibitem{Haa} 
U.~Haagerup, 
{\itshape Quasitraces on exact $C^*$-algebras are traces.} 
C. R. Math. Rep. Acad. Sci. Canada Vol. {\bf 36} (2-3) 2014, pp. 67--92. 

\bibitem{HO}
I. Hirshberg and J. Orovitz, 
{\itshape Tracially ${\cal Z}$-absorbing $C^*$-algebras.} 
J. Funct. Anal. {\bf 265} (2013), 765--785. 


\bibitem{LN} 
C.~Li and Z.~Niu. 
{\itshape Stable rank of $C(X)\rtimes \Gamma$.}
{{preprint, arXiv:2008.03361v2. 2020.}} 


\bibitem{Lin07}
H.~Lin,
{\itshape Simple nuclear $C^*$-algebras of tracial topological rank one.} 
J. Funct. Anal. {\bf 251} (2007), 601--679. 

\bibitem 
{LinJFA}
H.~Lin,  
{\itshape Strict comparison and stable rank one.} 
J. Funct. Anal. {\bf 289} (2025), no. 9, Paper No. 111065, 25 pp. 

\bibitem
{LinAdv}
H.~Lin,  
{\itshape Tracial oscillation zero and ${\cal Z}$-stability.} 
Adv. Math. {\bf 439} (2024), Paper No. 109462, 51 pp. 

\bibitem{Lor}
T.~A.~Loring, 
{\itshape Lifting solutions to perturbing problems in  $C^*$-algebras.} 
Fields Inst. Monogr., {\bf 8}.
American Mathematical Society, Providence, RI, 1997, x+165 pp.
ISBN: 0-8218-0602-5. 

\bibitem{Pedbk}  G. K. Pedersen, 
{\itshape  $C^*$-algebras and their automorphism groups.}
 London Mathematical Society Monographs, 14. Academic Press, Inc. London/New York/San Francisco, 1979.




\bibitem{RobRor}
L.~Robert and M.~R{\o}rdam,
{\itshape Divisibility properties for $C^*$-algebras.}
Proc. Lond. Math. Soc., 
vol.~{\bf 106}, no.~6 (2013), 1330--1370. 

\bibitem 
{RorUHF1}  
M.~R{\o}rdam, 
{\itshape On the structure of simple $C^*$-algebras tensored with a UHF-algebra.} 
J. Funct. Anal. {\bf 100} (1991), 1--17. 


\bibitem 
{RorUHF2}  
M.~R{\o}rdam, 
{\itshape On the structure of
simple $C^*$-algebras tensored with a UHF-algebra, II.} 
J.~Funct.~Anal. {\bf 107} (1992), 255--269.



\bibitem{T20}
H.~Thiel, 
{\itshape Ranks of operators in simple $C^*$-algebras with stable rank one.} 
Comm. Math. Phys. {\bf 377} (2020), no. 1, 37--76. 


\bibitem{T14} A.~Tikuisis, 
{\itshape Nuclear dimension, $\mathcal{Z}$-stability, and algebraic simplicity for stably projectionless $C^*$-algebras.}
Math.~Ann. (2014) {\bf 358}: 729--778. 


\bibitem{Vacc} A.~Vaccaro,
{\itshape Stable rank one, tracial local homogeneity and uniform property $\Gamma.$}
preprint, arXiv:2604.24682v2

\bibitem{V97}
J.~Villadsen, 
{\itshape Simple $C^*$-algebras with perforation.} 
J. Funct. Anal. {\bf 154} (1998), no. 1, 110--116.

\bibitem{W12pure} W.~Winter,
{\itshape Nuclear dimension and $\mathcal{Z}$-stability of pure $C^*$-algebras.}
Invent. Math. {\bf 187} (2012), no.~{2}, 259--342. 

\bibitem{WZ09OZ} W.~Winter and J.~Zacharias,
{\itshape Completely positive maps of order zero.} M{\"u}nster J. Math. {\bf 2} (2009), 311--324. 


\bibitem{Z91}
S.~Zhang,
{\itshape Matricial structure and homotopy type of simple 
$C^*$-algebras with real rank zero.} 
J.~Operator Theory {\bf 26} (1991), no.~2, 283--312. 

\end{thebibliography}
\end{document}